\documentclass[11pt]{article}

\usepackage{amsmath, amsfonts, amsthm, bbm}
\newcommand{\R}{\mathbb{R}}
\newcommand{\N}{\mathbb{N}}
\newcommand{\Z}{\mathbb{Z}}
\newcommand{\dd}{\mathrm{d}}
\newcommand{\E}{\mathbf{E}}
\newcommand{\p}{\mathbf{P}}
\DeclareMathOperator{\cov}{Cov}
\DeclareMathOperator{\var}{Var}

\newcommand{\toi}{\to\infty}
\newcommand{\eind}{\stackrel{d}{=}}
\newcommand\ind[1]{\mathbbm{1}{\left\{#1\right\}}}

\newcommand{\dto}{\stackrel{d}{\longrightarrow}}

\newcommand{\lo}{\tilde{l}_0}

\newcommand{\vep}{{\varepsilon}}

\newcommand{\dsum}{\displaystyle\sum}
\newcommand{\dint}{\displaystyle\int}

\theoremstyle{plain}
\newtheorem{lemma}{Lemma}
\newtheorem{theorem}{Theorem}
\newtheorem*{theorem*}{Theorem}

\newtheorem{corollary}{Corollary}

\theoremstyle{remark}
\newtheorem{remark}{Remark}
\newtheorem*{example}{Example}

\title{Limit Theorems for Branching Processes with \\ 
Immigration in a Random Environment}

\author{Bojan Basrak\thanks{Department of Mathematics, Faculty of Science,
University of Zagreb, Bijeni\v{c}ka 30, 10000 Zagreb, Croatia;
e-mail: \texttt{bbasrak@math.hr}}
\quad and \quad 
P\'eter Kevei\thanks{Bolyai Institute, University of Szeged, 
Aradi v\'ertan\'uk tere 1, 6720 Szeged, Hungary; 
e-mail: \texttt{kevei@math.u-szeged.hu}}}

\begin{document}

\maketitle

\begin{abstract}
We investigate subcritical Galton--Watson branching processes with immigration
in a random environment. Using Goldie's implicit renewal theory we show that under
general Cram\'er condition the stationary distribution has a power law tail.
We determine the tail process of the stationary Markov chain, prove 
point process convergence, and convergence of the partial sums.
The original motivation comes from Kesten, Kozlov and Spitzer seminal 1975 paper,
which connects a random walk in a random environment model to a special 
Galton--Watson process with immigration in a random environment.
We obtain new results even in this very special setting.

\end{abstract}

\section{Introduction and notation}

Kesten et al.~in their article \cite{KKS} on random walks in a random environment
discovered close  connections of such walks with the so-called stochastic 
recurrence equations, and with a  class of branching processes with immigration in a 
random environment. 
In this article, we aim to give a precise description of the long term behavior of such 
branching processes.
It is known that  in an i.i.d.~random environment, a branching process with 
immigration, $(X_n)$ say,  has a Markovian structure i.e.~it satisfies recursion
$ X_{n}  = \psi (X_{n-1},  Z_{n})$ for an i.i.d.~sequence of random elements $(Z_n)$. 
Although,  its evolution  seems somewhat more involved  than  the one described by the standard stochastic 
recurrence equations  (as studied in \cite{BDM} for instance), we  show in 
Section~\ref{sec:RVTails}  that  Goldie's implicit renewal theory of \cite{Goldie}  can be 
readjusted to characterize  the tails of the corresponding stationary distribution.
Next, in  Sections~\ref{sec:PointProc} and~\ref{sec:PartialSums} respectively, we  give 
detailed description of the asymptotic limits of properly normalized values $(X_i)$ for
$i=1,\ldots, n$, and of their partial sums $S_n=X_1+\cdots +X_n$, showing that the latter
can exhibit both Gaussian and non--Gaussian limits under appropriate conditions.
Note that Roitershtein~\cite{Roit07} 
showed  Gaussian asymptotics for partial sums of certain  multitype branching processes with 
immigration in a more general random environment. 
As an application of our main results, in Section~\ref{sec:RWRE}  we reconsider  Kesten 
et al.'s \cite{KKS} random walk in a random environment model with positive drift and 
present an alternative analysis of the asymptotic behavior of such a walk. We explain  how their
conditions yield various limiting distributions of such walks with the emphasis on the 
arguably more interesting non--Gaussian case.  Moreover,
we show that one can relax original conditions in \cite{KKS} somewhat, and also write out
the characteristic function of the limiting distributions. The precise form of the 
characteristic function which follows from our analysis seems to be new. Finally,  our 
method additionally yields the long term behavior of
the worst traps such a random walk encounters when moving to the right. It is already 
understood that it is exactly those traps, that is the edges visited over and over again 
as the walk  moves to the right, which give rise to the non--Gaussian limits for random walks in 
random environment.

Throughout the article, the random environment is modeled by an i.i.d. sequence  
$\xi$, $(\xi_t)_{t \in \Z}$ with values in a measurable space $\mathbb{X}$.
It may help in the sequel to specify  $\mathbb{X} = \Delta^2$, where $\Delta$
denotes the  space of probability measures on $\N= \{ 0, 1, \ldots \}$, 
in that case we may write $\xi = (\nu_{\xi},\nu^\circ_{\xi})$.
Alternatively, we assume that there exist a measurable function which maps each $\xi$ to
a pair $(\nu_{\xi},\nu^\circ_{\xi}) \in \Delta^2$.
The two components of the pair $(\nu_{\xi},\nu^\circ_{\xi})$ are called the 
offspring and the immigration distribution, respectively. The 
 Galton--Watson 
process with immigration in the random environment $\mathcal{E}=\sigma \{ 
\xi_t:t \in \Z\}$  evolves as follows. Let $X_0 = x \in \N$, and then set
\begin{equation} \label{eq:def-X}
X_{n+1} = \sum_{i=1}^{X_n} A_i^{(n+1)} + B_{n+1}
=: \theta_{n+1} \circ X_n + B_{n+1}, \quad n \geq 0, 
\end{equation}
where we assume that, conditioned on the environment $\mathcal{E}$, the variables 
$\{ A_i^{(n)}, B_n : \,  n \in \Z\,, i\geq 1 \}$ are  independent, moreover, 
for $n$ fixed, 
$(A_i^{(n)})_{i \geq 1}$ are i.i.d.~with distribution $\nu_{\xi_n}$, and $B_n$ has 
distribution $\nu^\circ_{\xi_n}$. We interpret the variable $A_i^{(n)}$ as the number of 
offspring 
of the $i$th element in the $(n-1)$st generation, 
and $B_n$ as the number of immigrants in the $n$th generation. Hence,
$\theta_n$ as in \eqref{eq:def-X} denote random operators on nonnegative integers which 
act as follows 
$$
\theta_{n} \circ x  = \sum_{i=1}^{x} A_i^{(n)}\,.
$$
In the sequel, we abuse the notation, by writing 
$\theta_n\circ (x+y) = \theta_{n} \circ x + \theta_{n} \circ y$, 
keeping in mind that the 
two random operators on the right hand side are not really the same.

For an $\mathbb{X}$--valued random element $\xi$  denote by
\[ %\begin{equation} \label{eq:m-def}
m(\xi) = \sum_{i=1}^\infty i  \nu_{\xi} (\{i\}), \quad 
m^\circ(\xi) = \sum_{i=1}^\infty i  \nu^\circ_{\xi} (\{i\}),
\] %\end{equation}
the expectation of its offspring and immigration distribution.
This is clearly a nonnegative random variable, potentially equal to $+\infty$.
We assume in the sequel that the process is subcritical, i.e.
\begin{equation} \label{eq:ass-subcrit}
\E \log m(\xi) < 0, 
\end{equation}
and that the following Cram\'er's condition holds
\begin{equation} \label{eq:ass-cramer}
\E m(\xi)^\kappa = 1  \quad  
%\text{ and  } \quad 
%\E   m(\xi)^\kappa  \log m(\xi) < \infty,
%\quad 
\text{for some } \kappa > 0. 
\end{equation}
Finally, to avoid trivial situations, we assume that 
$\nu_\xi^\circ$ is not concentrated at 0.

Note that by our assumptions  the  sequence of random elements $Z_n = (\xi_n, B_n, 
A^{(n)}_1,A^{(n)}_2,\ldots)$,\, ${n\in \Z}$ with values in $\mathbb{X} \times \N ^\N$
are iid.
Therefore, one can represent the evolution of the process $(X_n)$ using
a measurable mapping $\psi$ and
i.i.d.~sequence $(Z_n)$  as
\begin{equation}\label{eq:PsiX}
  X_{n+1}  = \psi (X_n,  Z_{n+1})   = \Psi_{n+1} (X_n)
  = \theta_{n+1} \circ X_n + B_{n+1}\,,
\end{equation}
emphasizing the Markovian character of the process $(X_n)$. Here  $\Psi_n$
denotes a random map $x \mapsto \psi (x,  Z_{n})$.

By iterating \eqref{eq:def-X}  backward one may expect that 
the stationary  distribution of the process can be 
found as
\begin{equation} \label{eq:def-Xstat}
X_\infty = B_0 + \theta_0 \circ B_{-1} + 
\theta_0 \circ \theta_{-1} \circ B_{-2} + \ldots 
\eind
\sum_{i=0}^\infty \theta_0 \circ \theta_1 \circ \ldots \circ \theta_{i-1}  \circ B_{i},
\end{equation}
provided that the r.h.s converges a.s.~to a finite limit 
(cf.~Lemma \ref{lemma:X-moments}).
Clearly, if well-defined, such a  distribution satisfies the distributional fixed point 
equation
\begin{equation} \label{eq:fixed}
X \eind \psi(X,Z)=  \sum_{i=1}^X A_i + B = \theta \circ X + B,
\end{equation}
where  $Z=(\xi,B, A_1,A_2,\ldots)$ and $X$ on the right-hand side are independent.

\section{Implicit renewal theory}\label{sec:RVTails}

\subsection{Moments of the stationary distribution}

One of the main steps in the analysis of the process $(X_n)$ is 
to determine whether the random variable $X_\infty$ in \eqref{eq:def-Xstat} has 
finite moments of order $\alpha>0$ say.  If this holds, $X_\infty$  would be clearly 
finite with probability one.
%Observe that under \eqref{eq:ass-cramer}, 
On the other hand, by the conditional Jensen inequality 
$\E m(\xi)^t \leq \E A^t$ for $t \geq 1$, while $\E m(\xi)^t \geq \E A^t$ for
$t \leq 1$. In particular, for any fixed $t>0$, $\E A^t = \infty$ is possible
with the assumptions
\eqref{eq:ass-subcrit} and \eqref{eq:ass-cramer} still being satisfied for some 
$\kappa>0$.

In deterministic environment for multitype processes, the
existence and explicit expression for the moments of order $\alpha$ 
were subject of Quine \cite{Quine} for $\alpha = 1,2$, and  of
Barczy et al.~\cite[Lemma 1]{BNP} for $\alpha = 3$. Under additional ergodicity conditions,
the corresponding result for multitype processes for general $\alpha > 0$
was proved   by Sz\H{u}cs \cite{szucs}. 
In our case, under the condition that $\alpha>0$
satisfies
\begin{equation} \label{eq:momCond}
\E m(\xi)^\alpha < 1\,,\quad \E A^\alpha < \infty\,, \quad 
\mbox{  and }\quad 
\E B^{\alpha} < \infty,
\end{equation}
by Lemma 3.1 in Buraczewski and Dyszewski \cite{BD} one can find
constants $c>0$ and $ 0< \varrho < 1$ such that 
\[
\E (\theta_0 \circ \theta_1 \circ \ldots \theta_{i-1} \circ B_{i})^{\alpha}
\leq c \varrho^i\,. 
\]
%%%%%%%\bonote{that BD lemma seems simple to prove, should we?}
Therefore, for $\alpha \geq 1$ by Minkowski's inequality
 \[
(\E X_\infty^\alpha)^{1/\alpha}  \leq  \sum_{i=0}^\infty 
\left( \E (\theta_0 \circ \theta_1 \circ \ldots \theta_{i-1} \circ B_{i})^{\alpha}  
\right)^{1/\alpha}
\leq c^{1/\alpha} \sum_{i=0}^\infty \varrho^{i/\alpha} < \infty,
\] 
while for $\alpha < 1$, simply by subadditivity
\[
\E X_\infty^\alpha \leq  \sum_{i=0}^\infty 
\E (\theta_0 \circ \theta_1 \circ \ldots \theta_{i-1} \circ B_{i})^{\alpha}  
\leq c \sum_{i=0}^\infty \varrho^{i} < \infty,
\] 
which immediately yields the following useful result.

\begin{lemma} \label{lemma:X-moments}
If \eqref{eq:momCond} holds for some $\alpha>0$, then the 
random variable $X_\infty$ in \eqref{eq:def-Xstat} satisfies
$\E X_\infty^\alpha < \infty$.
\end{lemma}

Lemma~\ref{lemma:X-moments} implies in particular that \eqref{eq:def-Xstat} represents a 
solution of   the distributional equation in  \eqref{eq:fixed}.
The next  statement can be viewed as a slight extension of Theorem 1.5.1 in Gut 
\cite{Gut}.
Put $a \vee b = \max \{ a, b\}$.

\begin{lemma} \label{lemma:random-sum-moments2}
Let $Y$ be a nonnegative integer-valued random variable, 
$A, A_1$, $A_2$, $\ldots$ are
identically distributed random variables, independent of $Y$, such that they are 
i.i.d.~given the environment $\xi$, $\E [A | \xi ] = 0$, furthermore $Y$ is 
independent of $\xi$. 
\begin{itemize}
\item[(i)] Let $\alpha \geq 2$. Assume that $\E Y^{\alpha/2} < \infty$, and
$\E |A|^\alpha < \infty$.
Then there is a $c = c(\alpha)  >0$ depending only on $\alpha$ such that
\[
\E \left| \sum_{i=1}^Y A_i \right|^{\alpha} < 
c \, \E Y^{\alpha /2} \, \E |A|^\alpha . 
\]

\item[(ii)] Let $\alpha \in (0, 2]$, $\eta \in [0,1]$ be such that
$2 \eta \leq \alpha \leq 1 + \eta$, and $\E Y^{\alpha - \eta} < \infty$,
and
\[
\E \left( \E \left[ |A|^{\frac{\alpha}{\alpha -\eta}} 
\Big| \xi \right] \right)^{\alpha -\eta} < \infty.
\]
Then there is a $c = c(\alpha, \eta)  >0$ depending only on $\alpha$
and $\eta$ such that
\[
\E \left| \sum_{i=1}^Y A_i \right|^{\alpha} < 
c \, \E Y^{\alpha-\eta} \, 
\E \left( \E \left[ |A|^{\frac{\alpha}{\alpha -\eta}} 
\Big| \xi \right] \right)^{\alpha -\eta}.
\]

\item[(iii)] For $\alpha \in [1,2]$ assume that $\E Y < \infty$ and
$\E |A|^\alpha < \infty$. Then there is a $c = c(\alpha) > 0$ such that
\[
\E \left| \sum_{i=1}^Y A_i \right|^\alpha \leq c \E Y \E |A|^\alpha.
\]
\end{itemize}
\end{lemma}

\begin{proof}
(i)
By Rosenthal's inequality \cite[Theorem 2.10]{Petrov}
\begin{equation} \label{eq:auxa1}
\begin{split}
\E \Big| \sum_{i=1}^n A_i \Big|^{\alpha} %&
 =
\E \bigg( 
\E \bigg[ \Big| \sum_{i=1}^n A_i \Big|^{\alpha} \bigg |\, \xi \bigg] 
\bigg) %\\&
 \leq c \, n^{\alpha/2} \, \E |A|^{\alpha}.
\end{split}
\end{equation}
From inequality \eqref{eq:auxa1} %and \eqref{eq:auxa2} 
the statement follows 
easily, as
\[
\begin{split}
\E \Big| \sum_{i=1}^Y A_i \Big|^{\alpha } 
& = \sum_{n=1}^\infty \p (Y = n ) \, \E \Big| \sum_{i=1}^n A_i \Big|^{\alpha} \\
& \leq c \, \E |A|^{\alpha} \, \sum_{n=1}^\infty \p (Y = n ) \; n^{\alpha/2} \\
& \leq c \, \E |A|^{\alpha} \, \E Y^{\alpha/2} < \infty.
\end{split}
\]

(ii) First using the conditional Jensen inequality 
($\alpha - \eta \leq 1$), then the Marcinkiewicz--Zygmund inequality
\cite[2.6.18]{Petrov},
and finally the subadditivity ($\alpha / [2 (\alpha -\eta)] \leq 1$)
we obtain
\[
\begin{split}
\E \Big| \sum_{i=1}^Y A_i \Big|^\alpha & =
\E \bigg(
\E \bigg[ \Big| \sum_{i=1}^Y A_i \Big|^\alpha \, \bigg|\, \xi , Y \bigg] 
\bigg) \\
& \leq 
\E \bigg(
\E \bigg[ \Big| \sum_{i=1}^Y A_i \Big|^{\frac{\alpha}{\alpha -\eta}}
\, \bigg| \, 
\xi , Y \bigg] 
\bigg)^{\alpha - \eta} \\
& \leq 
c \, \E \bigg(
\E \bigg[ \Big( \sum_{i=1}^Y A_i^2 \Big)^{\frac{\alpha}{2(\alpha -\eta)}}
\, \bigg|\,  \xi , Y \bigg] 
\bigg)^{\alpha - \eta} \\
& \leq c\,  \E Y^{\alpha - \eta} \,  
\E \left( \E \left[ |A|^{\frac{\alpha}{\alpha -\eta}} 
\, \Big| \xi \right] \right)^{\alpha -\eta}.
\end{split}
\]

(iii)  This follows from (ii) choosing $\eta = \alpha -1$.
\end{proof}

\subsection{Goldie's condition}

In \eqref{eq:PsiX} we described the evolution of the Markov chain $(X_n)$ using
an i.i.d.~sequence of random functions $(\Psi_n)$, 
$\Psi_n: \Omega \times \N \to \N$, $n \in \Z$, having the following general form
\[ %\begin{equation} \label{eq:def-psi}
\Psi_n( k ) = \sum_{i=1}^k A_i^{(n)} + B_n, \quad k \in \N.
\] % \end{equation}
Clearly, distributional  fixed point equation in \eqref{eq:fixed} can be written as
\begin{equation} \label{eq:fixed2}
X \eind \Psi(X),
\end{equation}
with $\Psi$ and $X$  independent on the right-hand side.

\begin{lemma} \label{lemma:Goldie-cond}
Assume that there exist $ \kappa > 0$  such that 
$ \E m(\xi)^\kappa = 1$, $\E A^\kappa < \infty$,
$\E B^{\kappa} < \infty$.
Then the law  in \eqref{eq:def-Xstat} represents the unique
stationary distribution for
the Markov chain $(X_n)$.
Suppose further that at least one of the following three conditions holds
\begin{itemize}
\item[(i)]  $\kappa > 1$  and $\E \left( m(\xi)^{\kappa-1} (\E [ A^2 | \xi ])^{1/2} 
\right) < \infty$;
\item[(ii)] $\kappa > 1$  and there exists  $\delta > 0$ such that 
$\E A^{ \kappa  + \delta} < \infty$;
\item[(iii)]  $\kappa \leq 1$ and there exists $\eta \in (0,\kappa/2]$
such that
\[
\E \left( \E \left[ A^{\frac{\kappa}{\kappa -\eta}} 
\,\Big|\, \xi \right] \right)^{\kappa -\eta} < \infty.
\]
\end{itemize}
Then the random variable  $X=X_\infty$  further satisfies
\[
\E \left| \Psi(X)^\kappa - (m(\xi) X)^\kappa \right| < \infty. 
\]
\end{lemma}

\begin{proof}
It is straightforward to see that the assumptions of the lemma imply \eqref{eq:momCond} 
for any $ \alpha \in (0, \kappa )$,
which, by Lemma \ref{lemma:X-moments} further proves that the Markov
chain in \eqref{eq:def-X} has a stationary 
distribution. Denote by
$$
d_0 =\min \{  k : \p ( A_1 = 0\,, B_1= k) >0 \}\,.
$$
The nonnegative integer $d_0$  is well defined since $\p(A_1=0) >0$ 
by the subcriticality assumption \eqref{eq:ass-subcrit}. Moreover, it
represents an accessible atom for the Markov chain $(X_n)$, this makes the chain 
irreducible, and the stationary distribution unique, see 
Douc et al.~\cite[Theorem 7.2.1]{DMS18} for instance. 

\smallskip

First consider the $\kappa > 1$ case.
We use that for any $\alpha \geq 1$ for some $c = c_\alpha > 0$
\[
\left| x^\alpha - y^\alpha \right| \leq 
c |x-y| \, ( y^{\alpha -1} + |x - y|^{\alpha -1} ). 
\]
Therefore
\begin{equation} \label{eq:aux1}
\begin{split}
& \E \left| \Psi(X)^\kappa - (m(\xi) X)^\kappa \right| \\
& \leq c \left( 
\E |\Psi(X) - m(\xi) X| (m(\xi) X)^{\kappa -1}  
+ \E |\Psi(X) - m(\xi) X|^\kappa   \right).
\end{split}
\end{equation}
For the second term in \eqref{eq:aux1} by Minkowski's inequality
\begin{equation} \label{eq:diff1}
\left( \E | \Psi(X) - m(\xi) X |^\kappa \right)^{1/\kappa} \leq 
\bigg( \E \Big| \sum_{i=1}^X (A_i - m(\xi)) \Big|^\kappa \bigg)^{1/\kappa} +
\left( \E B^\kappa \right)^{1/\kappa}.
\end{equation}
The second term is finite according to our assumptions. 
The finiteness of the first term in \eqref{eq:diff1} follows from Lemmas
\ref{lemma:X-moments} and \ref{lemma:random-sum-moments2}. Indeed, 
for $\kappa \geq 2$  Lemma \ref{lemma:random-sum-moments2} (i), 
for $\kappa \in (1,2)$  Lemma \ref{lemma:random-sum-moments2} (iii)
applies.

For the first term in \eqref{eq:aux1} we have
\begin{equation} \label{eq:aux2}
\begin{split}
& \E \left( |\Psi(X) - m(\xi) X| (m(\xi) X)^{\kappa -1} \right) \\
& \leq 
\E \bigg( 
\Big| \sum_{i=1}^X (A_i - m(\xi)) \Big| \, (m(\xi) X)^{\kappa -1} \bigg)
+ \E B (m(\xi) X)^{\kappa -1}.
\end{split}
\end{equation}
By independence and H\"older's inequality
\[ 
\begin{split}
\E B (m(\xi) X)^{\kappa -1} & = \E X^{\kappa -1} \E B m(\xi)^{\kappa -1} \\
& \leq \E X^{\kappa -1} (\E B^\kappa)^{1/\kappa} 
(\E m(\xi)^\kappa)^{(\kappa -1)/\kappa}
< \infty,
\end{split}
\]
so the second term in \eqref{eq:aux2} is finite.
Therefore, it only remains to show the finiteness of the first term 
in \eqref{eq:aux2}.

Assume (i). 
Applying first the Marcinkiewicz--Zygmund inequality and then 
Jensen's inequality
\[
\begin{split}
\E \bigg[ \Big| \sum_{i=1}^n (A_i - m(\xi)) \Big| \, \Big| \, \xi \bigg] 
& \leq 
c \,  \E \bigg[ \Big( 
\sum_{i=1}^n (A_i - m(\xi))^2 \Big)^{1/2}  \, \Big| \, \xi \bigg] \\
& \leq c \, 
\bigg( \E \Big[  \sum_{i=1}^n (A_i - m(\xi))^2   \, \Big| \, \xi \Big]
\bigg)^{1/2} \\
& \leq c \, n^{1/2} \, (\E [ A^2 | \xi])^{1/2}.
\end{split}
\]
Substituting back into the first term in \eqref{eq:aux2}
\[ %\begin{equation} \label{eq:aux3}
\begin{split}
& \E \bigg( 
\Big| \sum_{i=1}^X (A_i - m(\xi)) \Big| \, (m(\xi) X)^{\kappa -1} \bigg)
\\
& \leq 
c \E X^{\kappa -1/2} \E \left( m(\xi)^{\kappa-1} (\E [ A^2 | \xi])^{1/2} \right),
\end{split}
\] %\end{equation}
which is finite whenever (i) holds.

Assume (ii). 
For the first term in \eqref{eq:aux2} by from H\"older's 
inequality we have
\begin{equation} \label{eq:aux4}
\begin{split}
& \E \bigg( 
\Big| \sum_{i=1}^X (A_i - m(\xi)) \Big| 
(m(\xi) X)^{\kappa -1} \bigg) \\
& \leq \bigg( \E \Big| \sum_{i=1}^X (A_i - m(\xi)) \Big|^{p} \bigg)^{1/p}
\left( \E (m(\xi) X)^{q(\kappa -1)} \right)^{1/q},
\end{split}
\end{equation}
with $1/p + 1/q =1$. Choose $p = \kappa + \varepsilon$, for some 
$0 < \varepsilon < \delta$, with $\delta > 0$ given in the condition (ii).
Then easy computation shows that 
$q = \kappa / (\kappa -1) - \varepsilon'$, where $\varepsilon' \downarrow 0$ as 
$\varepsilon \downarrow 0$.
Since $q (\kappa -1) < \kappa$, the second factor is finite by the independence 
of $X$ and $\xi$, and by Lemma \ref{lemma:X-moments}.
The finiteness of the first factor in \eqref{eq:aux4} follows from Lemmas
\ref{lemma:X-moments} and \ref{lemma:random-sum-moments2}. Indeed, 
for $\kappa \geq 2$ this is immediate. For $\kappa \in (1,2)$ choose
$p = \kappa + \varepsilon \leq 2$ and apply Lemma 
\ref{lemma:random-sum-moments2} (iii).

\smallskip 
The case $\kappa \leq 1$ is simpler. By the inequality
\[
\left| x^\alpha - y^\alpha \right| \leq 
|x-y|^\alpha, 
\]
we have
\[
\left| \Psi(X)^\kappa - (m(\xi) X)^\kappa \right|
\leq 
|\Psi(X) - m(\xi) X|^\kappa. 
\]
Thus by subadditivity 
\[
\E \left| \Psi(X)^\kappa - (m(\xi) X)^{\kappa} \right| 
\leq 
\E \left| \sum_{i=1}^X (A_i - m(\xi)) \right|^\kappa 
+ \E B^\kappa.
\]
Since the second term is finite by assumption, 
it is enough to show that
\[
\E \left| \sum_{i=1}^X (A_i - m(\xi)) \right|^\kappa < \infty.
\]
This follows from Lemma \ref{lemma:random-sum-moments2} (ii),
with $\eta$ given in condition (iii).
\end{proof}

\begin{remark} \label{rem:PoiGeo}
For special classes of offspring distributions the conditions of the lemma can 
be simplified. If, 
\begin{equation} \label{eq:spec}
\E [ A^2 | \xi ] \leq c (m(\xi)^2 + 1) \quad \text{a.s.~for some $c > 1$},
\end{equation}
then both the condition for $\kappa \leq 1$ and condition (i) reduces to
$\E m(\xi)^\kappa < \infty$, which holds since $\E m(\xi)^\kappa = 1$.

In particular if, conditionally on $\xi$, $A$ has Poisson distribution 
with parameter $\lambda(\xi)> 0$, then $m(\xi)  = \E [ A | \xi ]  = \lambda(\xi)$ and 
$\E [ A^2 | \xi ] = \lambda(\xi)^2 + \lambda(\xi) \leq$ $2 ( m(\xi)^2 + 1)$.
While if $A$, conditionally on $\xi$,  has geometric distribution with 
parameter $p(\xi) \in (0,1)$, 
i.e.~$\p(A = k \mid \xi ) = (1-p(\xi))^k p(\xi)$, $k \geq 0$, then
$m(\xi)  = \E [ A | \xi ] = (1-p(\xi))/p(\xi)$ and 
$\E [ A^2 | \xi ] =(3-2p(\xi)+p(\xi)^2)/p(\xi)^2$ 
$\leq$ $3 (1-p(\xi))^2 / p(\xi)^2 + 3 
= 3 m(\xi) ^2 + 3$. Thus, in both cases \eqref{eq:spec} holds.
\end{remark}

From Goldie's Corollary 2.4 \cite{Goldie} we obtain the following.

\begin{theorem} \label{thm:nonarith}
Assume the conditions of Lemma \ref{lemma:Goldie-cond},
$\E m(\xi)^\kappa \log m(\xi) < \infty$,  and that
the law of $\log m(\xi)$ given $m(\xi) > 0$ is nonarithmetic. Then
the law of $X_\infty$ in \eqref{eq:def-Xstat} represents the unique
stationary distribution for
the Markov chain $(X_n)$ and 
 \begin{equation} \label{eq:Tail}
 \p ( X_\infty > x ) \sim C x^{-\kappa} \quad \text{as } x \to \infty,
 \end{equation}
where
\[
C = \frac{1}{\kappa \E m(\xi)^\kappa \log m(\xi)} \E \left[ 
\Psi(X_\infty)^\kappa - 
m(\xi)^\kappa X_\infty^\kappa \right]  > 0.
\]
\end{theorem}

\begin{proof}
According to Lemma~\ref{lemma:Goldie-cond},
the Markov chain $(X_n)$ has unique stationary distribution given in  
\eqref{eq:def-Xstat}.
Because the law of the immigrant distribution is not concentrated at 0, this
distribution is not trivial and therefore  there exists  
$d  =\min \{  k>0 : \p ( X_\infty= k) >0 \} >0$\,.
The state $d$ is necessarily positive recurrent, 
i.e.~$\E_d(\tau_d) < \infty$ where
$\tau_d = \min \{ n \geq 1: X_n = d \}$ denotes the return time to the
state $d$.

From Lemma~\ref{lemma:Goldie-cond}, we can conclude that the conditions
of Corollary 2.4 in 
\cite{Goldie} hold which yields the tail asymptotics in \eqref{eq:Tail}.
It remains to show the strict positivity of the constant $C$ above. 
One can deduce this from Afanasyev's Theorem 1 \cite{Afan2001} as
follows.

To the original process $X_n$, one can couple a process $Y_n$, starting
at $Y_0 = 1$, and satisfying
\[
Y_{n+1} = \sum_{i=1}^{Y_n } A_i^{(n+1)}, \quad n \geq 0,
\]
where the sequence of environments $(\xi_1, \xi_2, \ldots)$ are the same
as in \eqref{eq:def-X}. Thus,  $(Y_n)$ is a subcritical branching process
in random environment, such that by the construction $X_n \geq Y_n$
a.s.~for every $n \geq 0$ provided that we start $(X_n)$ at any state
different from 0. Theorem 1 in \cite{Afan2001}
states that 
\begin{equation} \label{eq:afan}
\lim_{x \to \infty} x^\kappa \p ( \sup_{n \geq 1} Y_n > x) = c > 0.
\end{equation}
Let $\tau_Y = \min \{ n \geq 1: Y_n = 0 \}$, clearly
$\tau_Y \leq \tau_d$. In particular, the process $Y_n$ dies out almost surely.
By the standard theory of Markov chains, see Theorem 7.2.1 in \cite{DMS18}
or  Theorem 10.4.9 in \cite{MeyTwe}, we have
\[
\begin{split}
\p ( X_\infty > x ) & = \frac{1}{\E_d \tau_d} \, 
\E_d \left[ \sum_{i=0}^{\tau_d - 1} \ind{ X_i > x}  \right] \\
& \geq \frac{1}{\E_d \tau_d} \,
\E \left[ \sum_{i=0}^{\tau_Y - 1} \ind{ Y_i > x } \,   \right] 
 \geq \frac{1}{\E_d \tau_d} \,
\p ( \sup_{n \geq 1} Y_n > x),
\end{split}
\]
which by \eqref{eq:afan} implies
\[
\liminf_{x \to \infty} x^\kappa \p (X_\infty > x) > 0,
\]
as claimed.
\end{proof}

\begin{remark}
As a consequence of the applied machinery,
apart from some special cases, the constant $C$ in the 
theorem is merely implicit; the formula contains the limit law itself.
However, for $\kappa = 1$ we simply have 
$\E [ \Psi(X_\infty) - m(\xi) X_\infty] = \E B$, so $C$ can be 
computed. 
% We also note that for $\kappa \geq 1$ Afanasyev's result is not 
% necessary. Indeed, simple application of 
% Jensen's inequality shows the strict positivity
% of the constant $C$.
\end{remark}

\begin{remark}
Analogous results hold in the arithmetic case.
Using Theorem 2 in \cite{Kevei2} (see 
also Theorem 3.7 by Jelenkovi\'c and Olvera-Cravioto \cite{JOC}) one can
show the  following. If the conditions of Lemma \ref{lemma:Goldie-cond} hold, and 
the law of $\log m(\xi)$ given $m(\xi) > 0$ is arithmetic with span 
$h > 0$,  then there exists a function $q$ such that 
\begin{equation*} 
\lim_{n \to \infty} x^\kappa e^{\kappa nh}  
\p ( X_\infty > x e^{ nh} ) = q(x),
\end{equation*}
whenever $x$ is a continuity point of $q$.
The function $x^{-\kappa} q(x)$ is nonincreasing, and $q(x e^h) = q(x)$ for all 
$x > 0$.
%Moreover, $q(x) > 0$ for all $x>0$. %%% who knows?
\end{remark}

\subsection{Relaxing Cram\'er's condition}

In what follows, we weaken condition \eqref{eq:ass-cramer} in two ways,
such that the tail of the stationary distribution is still regularly varying.
We use a slight extension of Goldie's renewal theory by Kevei \cite{Kevei}.

First we consider weakening the assumption
$\E m(\xi)^\kappa \log m(\xi) < \infty$.
The condition $\E m(\xi)^\kappa = 1$ ensures that 
\[ %\begin{equation} \label{eq:def-Fk}
F_\kappa(x) = \E ( \ind{  \log m(\xi ) \leq x} m(\xi)^\kappa )
\] %\end{equation}
is a distribution function.
The additional logarithmic moment condition in \eqref{eq:ass-cramer} is equivalent
to the finiteness of the expectation of the distribution $F_\kappa$. This condition 
is needed to use the standard key renewal theorem. However, strong renewal
theorems in the infinite mean case have been known since the 1962 
paper by Garsia and Lamperti \cite{Garsia}. They showed that the strong renewal
theorem holds if the underlying 
distribution belongs to the domain of attraction of an $\alpha$-stable law 
with $\alpha \in (1/2,1]$, while for $\alpha \leq 1/2$ extra conditions are needed.
Recently, Caravenna and Doney \cite{Caravenna}
obtained necessary and sufficient conditions for the strong 
renewal theorem to hold, solving a 50-year old open problem.

Assume that $F_\kappa$ belongs to the domain of attraction of a stable 
law of index $\alpha \in (0,1]$, i.e., 
for some $\kappa > 0$ and $\alpha \in (0,1]$, for a slowly varying function $\ell$
\begin{equation} \label{eq:cramer-2}
\E m(\xi)^\kappa =1, \quad 
1 - F_\kappa (x) =  \frac{\ell(x)}{x^\alpha}.
\end{equation}
Define the truncated expectation as
\begin{equation} \label{eq:M-def}
M(x) = \int_0^x ( 1 - F_\kappa(x) ) \dd x \sim \frac{\ell(x) x^{1-\alpha}}{1-\alpha},
\end{equation}
where the asymptotic equality holds for $\alpha < 1$.
If $\alpha \in (0,1/2)$ further assume the Caravenna--Doney condition 
\begin{equation} \label{eq:CarDon-cond}
\lim_{\delta \to 0} \limsup_{x \to \infty} x [ 1 - F_\kappa (x) ]
\int_1^{\delta x} \frac{1}{y [1-F_\kappa(y)]^2} F_\kappa(x - \dd y) = 0.
\end{equation}

\medskip

For our second extension, 
assume now that $\E m(\xi)^\kappa = \varphi < 1$ for some $\kappa > 0$.
If $\E m(\xi)^t < \infty$ for some $t > \kappa$, then by Lemma \ref{lemma:X-moments}
the tail of $X_\infty$ cannot be regularly varying with index $\kappa$.
Therefore we assume $F_\kappa$ is heavy-tailed, i.e.~$\E m(\xi)^t = \infty$
for any $t > \kappa$. Define now the distribution function
\begin{equation} \label{eq:def-Fk2} 
F_\kappa (x) = \varphi^{-1}  
\E (\ind{\log m(\xi) \leq x} m(\xi)^\kappa ).
\end{equation}
The analysis of the stochastic fixed point equation \eqref{eq:fixed} leads to a 
defective renewal equation. To understand the asymptotic behavior of the solution
of these equations we need to introduce \emph{locally subexponential distributions}.

For $T \in (0,\infty]$ let $\nabla_T = (0, T]$ and  for a 
distribution function (df) $H$ we put
$H(x + \nabla_T) = H(x+T) - H(x)$. Let $*$ denote the usual convolution operator.
A df $H$ is \emph{locally subexponential}, $H \in \mathcal{S}_{loc}$, if
for each $T \in (0,\infty]$ we have
(i) $H(x+t+\nabla_T) \sim H(x+\nabla_T)$ as $x \to \infty$ uniformly in 
$t \in [0,1]$,
(ii) $H(x + \nabla_T) > 0$ for $x$ large enough, and
(iii) $H*H(x + \nabla_T) \sim 2 H(x + \nabla_T)$ as $x \to \infty$.
For more details see Foss, Korshunov and Zachary \cite[Section 4.7]{FKZ}.
Informally, a locally subexponential distribution is a subexponential 
distribution with well-behaved density function.

Our assumptions on $m(\xi)$ are the following:
\begin{equation} \label{eq:m-ass<1}
\begin{gathered}
\E m(\xi)^\kappa = \varphi < 1, \  \kappa > 0, \quad
F_\kappa \in \mathcal{S}_{loc}, \\
\text{for each $T \in (0,\infty]$} \ \sup_{y > x} F_\kappa(y + \nabla_T) 
= O(F_\kappa(x+\nabla_T)) \text{ as $x  \to \infty$.}
\end{gathered}
\end{equation}

\begin{theorem} \label{thm:nonarith2}
Assume that condition (i) or (iii) in Lemma \ref{lemma:Goldie-cond} holds,
the law of $\log m(\xi)$ given $m(\xi) > 0$ is nonarithmetic,
$\E B^\nu < \infty$ for some $\nu > \kappa$, and one of the following
two conditions is satisfied:
\begin{itemize}
\item[(i)] 
condition \eqref{eq:cramer-2} holds, and if $\alpha \in (0,1/2)$ also 
\eqref{eq:CarDon-cond} holds;

\item[(ii)]
condition \eqref{eq:m-ass<1} holds.
\end{itemize}
Then the law of $X_\infty$ in \eqref{eq:def-Xstat} represents the unique
stationary distribution for
the Markov chain $(X_n)$ and 
\[ %\begin{equation} \label{eq:Tail-2}
\p ( X_\infty > x ) \sim C x^{-\kappa} L(x) \quad \text{as } x \to \infty,
\] % \end{equation}
where
\[
L(x) = 
\begin{cases}
(\Gamma(\alpha) \Gamma(2-\alpha) M({\log x}))^{-1}, & \text{ in case (i),} \\
(F_\kappa(1 + \log x) - F_\kappa(\log x)) \varphi / ( 1- \varphi)^2, & 
\text{ in case (ii),}
\end{cases}
\]
is a slowly varying function, and
\[
C = \frac{1}{  \kappa} 
\E \left[ \Psi(X_\infty)^\kappa - 
m(\xi)^\kappa X_\infty^\kappa \right] \geq 0,
\]
with $C > 0$ for $\kappa \geq 1$.
\end{theorem}

\begin{proof}
The result follows from Theorems 2.1 and 2.3 in \cite{Kevei}. We only have to
check that
\[
 \E \bigg| \Big( \sum_{i=1}^{X_\infty} A_i + B \Big)^{\kappa + \delta} - 
 (m(\xi) X_\infty)^{\kappa + \delta} \bigg| < \infty
\]
for some $\delta > 0$. This can be done exactly the same way as in the proof of
Lemma \ref{lemma:Goldie-cond} case (i) and (iii).

The strict positivity  of the constant $C$ in the theorem  can be deduced directly 
from its form. Indeed, since $X_\infty$ is 
independent of $\Psi$ and $\xi$
\begin{equation} \label{eq:c1}
\begin{split}
\E \left[ \Psi(X_\infty)^\kappa - m(\xi)^\kappa 
X_\infty^\kappa \right] 
& = \sum_{n=1}^\infty \p ( X_\infty = n ) \E \bigg[
\Big( \sum_{i=1}^n A_i + B \Big)^\kappa - m(\xi)^\kappa n^\kappa \bigg] \\
& > \sum_{n=1}^\infty \p ( X_\infty = n ) 
\E \bigg[ \Big( \sum_{i=1}^n A_i \Big)^\kappa 
- m(\xi)^\kappa n^\kappa \bigg],
\end{split}
\end{equation}
where in the last step we used that $B$ is not identically 0.
By Jensen's inequality
\[
\begin{split}
\E \left[ \left( \sum_{i=1}^n A_i \right)^\kappa - 
m(\xi)^\kappa n^\kappa \right]
& =  \E \left( \E \left[ \left( \sum_{i=1}^n A_i \right)^\kappa \bigg| \xi 
\right]
- m(\xi)^\kappa n^\kappa \right)  \\
& \geq \E \left( n^\kappa m(\xi)^\kappa - m(\xi)^\kappa n^\kappa
\right) = 0,
\end{split}
\]
from which the strict positivity follows. 
\end{proof}

\begin{remark}
Note that in general we only proved the strict positivity of $C$ above for 
$\kappa \geq 1$. However, in some special cases it is possible 
to show that $C > 0$ for a general $\kappa>0$. In particular, if 
the expectation in the infinite sum in the first line of \eqref{eq:c1}
is strictly positive for each $n$, then clearly $C > 0$.
\end{remark}

\begin{remark}
Note that both remarks after Theorem \ref{thm:nonarith} 
apply in this setup as well.
\end{remark}

\begin{remark}
The condition $F_\kappa \in \mathcal{S}_{loc}$ 
is stronger than the regular variation condition, in particular
Pareto distribution is locally subexponential. It is known that lognormal and 
Weibull distributions belong to the class $\mathcal{S}_{loc}$ as well.
In the Pareto case, i.e.~if for large enough $x$ we 
have $1 - F_\kappa(x) = c \, x^{-\beta}$, for some $c > 0, \beta > 0$,
then $\p ( X_\infty > x ) \sim c' x^{-\kappa} (\log x)^{-\beta -1}$.
In the lognormal case, when $F_\kappa(x) = \Phi(\log x)$ for $x$ large enough,
with $\Phi$ being the standard normal df, we have
$\p ( X_\infty > x ) \sim c x^{-\kappa} e^{-(\log \log x)^2/2} / \log x$, $c > 0$.
While, for Weibull tails
$1 - F_\kappa(x) = e^{-x^\beta}$, $\beta \in (0,1)$, we obtain
$\p ( X_\infty > x ) \sim c x^{-\kappa} (\log x)^{\beta-1} e^{-(\log x)^\beta}$,
$c > 0$. 
\end{remark}

\section{Asymptotic behavior of the process $(X_t)$} \label{sec:PointProc}

\subsection{Dependence structure of the process}

%In the sequel we assume that the conditions 
%of Theorem~\ref{thm:nonarith} hold.

In the sequel we only need that our process has a stationary distribution
with regularly varying tail. In the previous section
we derived several conditions which ensures regularly varying tail.
Since in Theorem \ref{thm:nonarith2} we only have the strict positivity of the
underlying constant $C$ for $\kappa \geq 1$, we assume that one of the following
holds:
\begin{equation} \label{eq:main-ass}
\begin{gathered}
\text{conditions of Theorem \ref{thm:nonarith},} \quad 
 \text{or conditions of  Theorem \ref{thm:nonarith2} 
and $C > 0$.}
 %for some $\kappa \geq 1$.}
\end{gathered}
\end{equation}
Then in particular, there exist a strictly stationary Galton--Watson 
process with immigration in random environment $(X_n)_{n \in \Z}$ which satisfies 
\begin{equation} \label{eq:def-XZ}
X_{n+1} = \sum_{i=1}^{X_n} A_i^{(n+1)} + B_{n+1}, \quad n \in \Z\,,
\end{equation}
with the same interpretation of the random  variables 
$\{ A_i^{(n)}, B_n : \, n \in \Z\,, i \geq 1\}$.
Here  again by $\xi$, $(\xi_i)_{i \in \Z}$ we denote i.i.d.~random variables
representing the environment. The offspring and 
immigration distributions are governed by the environment $\mathcal{E}$
in the same way as before.

It is useful in the sequel to introduce a deterministic sequence $(a_n)_{n \in \N}$ 
such that
\begin{equation} \label{eq:def-a} 
n \p (X_\infty > a_n) \longrightarrow 1 \quad \text{as } n \to \infty.
\end{equation}
Observe that by Theorem~\ref{thm:nonarith} and \ref{thm:nonarith2},
$(a_n)$ is a regularly varying sequence with index $1/\kappa$, 
i.e.~$a_n = \widetilde \ell(n) n^{1/\kappa}$, with an appropriate 
slowly varying function $\widetilde \ell$. In particular, 
if the conditions of Theorem \ref{thm:nonarith} hold 
we may set $a_n= (Cn)^{1/\kappa}$ and that  
any other sequence $(a_n)$ in \eqref{eq:def-a} necessarily satisfies
$a_n \sim (Cn)^{1/\kappa}$.

Once we have shown that the marginal stationary distribution of the $X_t$'s is
regularly varying, it is relatively easy to prove that all the finite
dimensional distributions of $(X_t)$ in 
\eqref{eq:def-XZ} have multivariate regular variation property, 
cf.~\cite{resnick:2007}, i.e.~$(X_n)_{n \in \Z}$ is regularly varying
sequence in the sense of \cite{basrak:segers:09}. According to 
Theorem 2.1 in \cite{basrak:segers:09}, this is equivalent to the existence
of the so-called tail sequence, which is the content of the first theorem below.

Observe first that under Cram\'er's condition \eqref{eq:ass-cramer}, one can construct  
a tilted distribution of the following form
\[ 
\p (\xi ^* \in \cdot\; )  =
\E \left(  \ind{\xi \in \cdot} m(\xi)^ \kappa \right)\,.
\]
Similar change of measure appeared already in Afanasyev \cite[Lemma 1]{Afan2001}, see 
also \cite[Lemma 3.1]{BD}. 
By the convexity of the function
$\lambda(\alpha) = \E m(\xi)^\alpha$
\[
\E \log m(\xi^*) = \E m(\xi)^\kappa \log m(\xi) 
= \lambda'(\kappa) > 0,
\]
possibly infinite,
that is, the branching after the change of measure becomes  supercritical.

To simplify notation, denote $m= m (\xi),\, m_i
= m(\xi_i)$, $i \geq 1$, under their original  distribution.
Introduce further an auxiliary i.i.d.~sequence
$m^* = m(\xi^*)$, $m^*_i = m(\xi^*_i)$, $i \geq 1$, with the 
common distribution 
$\p (m^* \in A) = \E [\ind{m(\xi) \in A } m(\xi)^\kappa ]$
and independent of $(m_i)$.

On the other hand, if \eqref{eq:m-ass<1} holds, 
i.e.~$\E m(\xi)^\kappa = \varphi < 1$, then 
$m^*_i$, $i \geq 1$, is a sequence of 
\emph{extended random variables} with 
common distribution 
$\p (m^* \in A) = \E [\ind{m(\xi) \in A } m(\xi)^\kappa ]$,
for $A \subset \R$, and $\p (m^* = \infty) = 1 -\varphi$,
and independent of $(m_i)$.
In the following result we use the usual convention $1/\infty = 0$.

\begin{theorem} \label{thm:tail-proc}
Let $(X_n)_{n \in \Z}$ be a  
strictly stationary sequence satisfying \eqref{eq:def-XZ}.
Assume \eqref{eq:main-ass}, and 
let $(m_n)_{n \geq 1}$, $(m^*_n)_{n \geq 1}$ be independent 
i.i.d.~sequences introduced above and independent of $Y_0$ with Pareto distribution
$\p(Y_0 > u) = u ^{-\kappa}$\,, $u \geq 1$.
Then, for any integers $k, \ell \geq 0$
\begin{eqnarray*}
\lefteqn{\mathcal{L} \left( 
\frac{X_{-k}}{x}, \ldots, \frac{X_0}{x}, \ldots, \frac{X_{\ell}}{x} \, 
\Big| \,
X_0 > x \right) } \\
& \qquad \stackrel{d}{\longrightarrow} 
Y_0 \left(  ({m^*_{k}\cdots m^*_1})^{-1}, \ldots, m_1 ^{*-1}, 1, m_1,
\ldots,  m_1 \cdots m_{\ell} \right).
\end{eqnarray*}
\end{theorem}

Writing the random vector on the r.h.s.~above as  
$(Y_t,\,  t =-k, \ldots ,l)$, note that,
in the language of Basrak and Segers \cite{basrak:segers:09}, 
the sequence $(Y_t)_{t\in \Z}$ 
represents the tail process of the sequence $(X_t)$. Hence, in this case,
both the forward and backward tail processes are multiplicative random walks.

\begin{proof}
We will first show that for arbitrary $\ell \geq 0$
and  $\vep >0$  
\begin{equation} \label{eq:CCinP}
\p \left( \left| \frac{X_{\ell}}{X_0} - m_1 \cdots m_{\ell} \right|  > \vep \, 
\Big| \,
X_0 > x \right) \longrightarrow 0\,, \mbox{  as } x \toi\,.
\end{equation}
Indeed, consider first $\ell =1$,  
by the independence of $ A^{(1)}_j$'s and $B_1$ from 
$X_0$, 
\[
\begin{split}
\lefteqn{\p \left( \left| \frac{X_{1}}{X_0} - m_1  \right|  > \vep \, 
\Big| \, X_0 > x \right)  } \\
& =  \dsum_{k > x} \p \left( \left| 
\frac{ \sum_{j=1}^{X_0}  A^{(1)}_j + B_1 }{X_0} - m_1  \right|  > 
\vep \,;\, X_0 = k \right)
 \frac{1}{\p \left( X_0 > x \right)}\,. \\
& =  \dsum_{k > x} \p \left( \left| 
\frac{ \sum_{j=1}^{k}  A^{(1)}_j + B_1 }{k} - m_1 \right|  > \vep 
\right)
\frac{\p \left( X_0 = k \right) }{\p \left( X_0 > x \right)}\,. 
\end{split}
\]
By the ergodic theorem 
$ n^{-1} \sum_{j=1}^n A^{(1)}_j \to m_1$ a.s. Hence the r.h.s.~above tends to  0 as 
$x \toi$. Instead of general $\ell \geq 1$, consider for simplicity
$\ell = 2$. One can write
$X_2= \sum_{j=1}^{X_0} \tilde{A}^{(2)}_j  + 
\sum_{j=1}^{B_1} {A}^{(2)}_j  + B_2$. Given the environment 
$\mathcal{E}$, the three terms on the r.h.s.~are independent and 
$\tilde {A}^{(2)}_j$'s are i.i.d.~with the following distribution
\[
\tilde{A}^{(2)}_1
\eind \sum_{j=1}^{{A}^{(1)}_1 } {A}^{(2)}_j \,.
\]
In particular,  $ \E  [ \tilde{A}^{(2)}_1 \, | \, \mathcal{E} ] = m_1 m_2$, 
the ergodic theorem applied on the sequence 
$(\tilde{A}^{(2)}_j)$ again yields
\[ 
\begin{split}
& {\p \left( \left| \frac{X_{2}}{X_0} - m_1 m_{2} \right|  > \vep 
\, \Big| \, X_0 > x \right)  } \\
& =  \dsum_{k > x} \p \left( \left| \frac{ \sum_{j=1}^{k}  
\tilde{A}^{(2)}_j   + \sum_{j=1}^{B_1} {A}^{(2)}_j + B_2 }{k} - m_1 m_2 \right|  > \vep 
\right)
\frac{\p \left( X_0 = k \right) }{\p \left( X_0 > x \right)}\,
\to 0
\end{split}
\]
as $x \toi$. The same argument works for any $\ell \geq 1$, which implies 
\eqref{eq:CCinP}. Observing that $\mathcal{L} ({X_0}/{x} \mid X_0> x ) \dto Y_0$, 
we can conclude that for an arbitrary $\ell$ 
\begin{eqnarray*}
{\mathcal{L} \left( 
\frac{X_0}{x}, 
\frac{X_{1}}{X_0}, \ldots, \frac{X_{\ell}}{X_0} \, 
\Big| \,
X_0 > x \right) } 
 \stackrel{d}{\longrightarrow} 
\left( Y_0,  m_1,\ldots, m_1\cdots m_{\ell}  \right).
\end{eqnarray*}
%\bonote{using conditional Slutsky here, but ok}
This proves the statement of the theorem for $k=0$. 
In particular, the sequence $(X_t)$ is 
regularly varying, and the multiplicative random walk 
\begin{equation}\label{eq:TailFor}
Y_i  = Y_0 \Theta_i = Y_0 m_1\cdots m_i\,,\qquad i \geq 0 
\end{equation} 
represents the forward part of its tail process.  
By Theorem 3.1 in \cite{basrak:segers:09},
this uniquely determines the distribution of the whole tail process, including the 
negative indices.
The past distribution of the tail process has been determined already 
(see Theorem 5.2 and Example 3.3 in Segers~\cite{segers:07} for instance). 
It turns out that  the backward part of the tail 
process has the representation
\begin{equation}\label{eq:TailBack}
Y_{-k}  = Y_0 \Theta_{-k} = Y_0 / (m^*_k\cdots m^*_1)\,,\qquad k > 0 
\end{equation}
for  an i.i.d.~sequence $(m^*_n)$ as in the statement of the theorem.
\end{proof}

Alternatively, the theorem above can be obtained  using the general 
results by Janssen and Segers \cite[Theorem 2.1]{janssen:segers:14}. 

\smallskip

Next we prove that the large values in the sequence $(X_n)$ cannot linger
for "too long", i.e.~the so called anticlustering condition from \cite{DH95} holds.

\begin{lemma} \label{lemma:anticluster}
Let $r_n$ be a sequence such that $r_n = o(n)$. Then for any $u>0$
\[
\lim_{k \to \infty} \limsup_{n \to \infty} \p 
\left( \max_{k \leq |t| \leq r_n} X_t > a_n u \, \bigg|\, X_0 > a_n u  \right)  = 0.
\]
\end{lemma}

\begin{proof}
The proof is similar to the proof of Lemma 3.2 in \cite{BKP}. By the strict stationarity
\[
\p \left( \max_{k \leq |t| \leq r_n} X_t > a_n u \, \mid\,  X_0 > a_n u  \right) \leq 
2 \sum_{t=k}^{r_n} \p ( X_t > a_n u  \, \mid\,  X_0 > a_n u).
\]
Consider the decomposition
\begin{equation} \label{eq:stat-decomp}
X_t = \sum_{i=1}^{X_0} \widetilde A_i^{(t)} + \sum_{k=0}^{t-1} C_{t,k} ,
\end{equation}
where the first sum stands for the descendants of generation 0, while the second sum 
stands for the descendants of the immigrants arrived after generation 0,
i.e.~$C_{t,k} = \theta_t \circ \cdots \circ \theta_{t-k+1} \circ B_{t-k}$ 
denotes the number of descendants in generation $t$ from
immigrants arrived in generation $t-k$.
Then $X_0$ and 
$(C_{t,k})_{k=0,1,\ldots, t-1}$ are independent, and 
$(\widetilde A_i^{(t)})_{i \in \N}$ are independent given the environments 
$\xi_1, \ldots, \xi_t$. 
Then
\[ %\begin{equation} \label{eq:anticlust-1}
\begin{split}
\p ( X_t > a_n u | X_0 > a_n u ) & \leq 
\p \left( \sum_{k=0}^{t-1} C_{t,k}  > a_n u/2  \, \big| \,  X_0 > a_n u \right) \\
& \quad + \p \left( \sum_{i=1}^{X_0} \widetilde A_i^{(t)}
> a_n u/2 \, \big| \,  X_0 > a_n u \right).
\end{split}
\] % \end{equation}
Using the independence of $X_0$ and 
$(C_{t,k})_{k=0,1,\ldots, t-1}$
\[
\begin{split}
& \p \left( \sum_{k=0}^{t-1} C_{t,k} > a_n u / 2 \, \bigg| \, X_0 > a_n u \right) 
= \p \left( \sum_{k=0}^{t-1} C_{t,k} > a_n u / 2 \right) \\
& \leq  \p \left( \sum_{k=0}^{\infty} C_{t,k} > a_n u / 2 \right) 
 = \p \left( X_0 > a_n u / 2 \right).
\end{split}
\]
By \eqref{eq:def-a}
\[
\limsup_{n \to \infty} r_n  \p \left( X_0 > a_n u / 2 \right) = 0,
\]
so it remains to show that
\[ % \begin{equation} \label{eq:anticlust-2}
\lim_{k \to \infty} \limsup_{n \to \infty}
\sum_{t=k}^{r_n} \p \left( \sum_{i=1}^{X_0} \widetilde A_i^{(t)} 
> a_n u / 2 \, \Big| \, X_0 > a_n u \right) = 0. 
\] % \end{equation}
Choose $\alpha < \min\{ \kappa ,1 \}$. Then, by the convexity of the function
$u \mapsto \E m(\xi)^u$ we obtain $\E m(\xi)^\alpha < 1$. 
Note that
\[
\E [ \widetilde A_i^{(t)} | \xi_1, \ldots, \xi_t ] 
= m(\xi_1) \ldots m(\xi_t).
\]
Therefore,% as in \eqref{eq:Xmom-aux},
\[
\begin{split}
& \E \left[ \left( 
\frac{\sum_{i=1}^k  \widetilde A_i^{(t)} }
{k  m(\xi_1) \ldots m(\xi_t)}
\right)^\alpha 
\Big|  \xi_1, \ldots, \xi_t \right] \\
& \leq 
\E \left[  
\frac{\sum_{i=1}^k  \widetilde A_i^{(t)} }
{k  m(\xi_1) \ldots m(\xi_t)}
\ind{\sum_{i=1}^k  \widetilde A_i^{(t)} >  
k  m(\xi_1) \ldots m(\xi_t)} + 1  
\Big|  \xi_1, \ldots, \xi_t \right] \leq 2,
\end{split}
\]
which implies that
\[
\begin{split}
\E \left( \sum_{i=1}^k \widetilde A_i^{(t)} \right)^\alpha 
& \leq 2 k^\alpha 
\E \left(  m(\xi_1)^\alpha \ldots m(\xi_t)^\alpha \right) 
% \\ & 
= 2 k^\alpha \left[ \E m(\xi)^\alpha \right]^t.
\end{split}
\]
Note that the upper bound is summable in $t$, thus, by Markov's inequality
\[
\begin{split}
& \sum_{t=k}^{r_n} \p \left( \sum_{i=1}^{X_0} \widetilde A_i^{(t)} 
> a_n u / 2 \, \Big| \, X_0 > a_n u \right) \\
& = \frac{1}{\p ( X_0 > a_n u)}
\sum_{t=k}^{r_n} \sum_{\ell > a_n u}
\p ( X_0 = \ell) \, 
\p \left( \sum_{i=1}^{\ell} \widetilde A_i^{(t)} > a_n u / 2 \right) \\
& \leq 
\frac{1}{ \p ( X_0 > a_n u)}
\sum_{t=k}^{r_n} \sum_{\ell > a_n u}
\frac{2^{1+\alpha}}{(a_n u)^\alpha} \ell^\alpha 
\left[ \E m(\xi)^\alpha \right]^t \\
& \leq 
2^{1+\alpha} \frac{ \E X_0^\alpha \ind{ X_0 > a_n u}}
{(a_n u)^\alpha  \p ( X_0 > a_n u) }
\sum_{t=k}^{\infty} \left[ \E m(\xi)^\alpha \right]^t .
\end{split}
\]
As $n \to \infty$
the second factor here tends to $\kappa / (\kappa - \alpha)$ by Karamata's theorem,
while the third factor vanishes as $k \to \infty$.
Thus the result follows.
\end{proof}

Recall that by our assumptions on progeny and immigrant distributions,
there exists  $\alpha$, $0 < \alpha < \kappa$, $\alpha \leq 1$,
such that
\[ % \begin{equation}\label{eq:alphaAB}
\E m(\xi)^\alpha < 1 \quad  \mbox{  and  }\quad 
\E m^\circ(\xi)^{\alpha} < \infty.
\] % \end{equation}

Denote the Markov transition kernel of the sequence $(X_n)$ by $P(x, \cdot )$ and
consequently, by $ P^n(x, \cdot)$, $n \geq 1$ denote the $n$-step Markov transition
kernel corresponding to $P$.
Next we apply the standard drift method to show that the Markov chain $(X_n)$  is 
uniformly $V$-geometrically ergodic. First we introduce some notation.
For a function $V: \N \to [1,\infty)$, and any two probability measures 
$\nu_1, \nu_2 \in \Delta$ put 
\[
\| \nu_1  - \nu_2 \|_V = 
\sup_{g: \, |g| \leq V} \Big| \sum_{n=0}^\infty g(n) 
\left( \nu_1(\{n\}) - \nu_2(\{n\}) \right) \Big|.
\]
We use the notation $\E_\zeta$ when the initial distribution of $X_0$ is
$\zeta$, in particular, $\E_x$ means that $X_0 = x$. The stationary 
distribution is denoted by $\pi$.

\begin{lemma} \label{lemma:X-ergod}
Let $V(x) = x^\alpha + 1$, with $\alpha < \kappa$, $\alpha \leq 1$.
The Markov chain  $(X_n)_{n}$ is uniformly $V$-geometrically ergodic, 
that is, there exists $\rho \in (0,1)$ and $C > 0$ 
such that for each $x \in \N$
\[
\| P^n ( x,  \cdot )  - \pi  \|_{V} \leq C V(x) \rho^n\,,
\]
where $\pi$ denotes the unique stationary distribution of the Markov chain.
\end{lemma}

\begin{proof}
Applying Theorem 16.0.1 in \cite{MeyTwe} (equivalence of (ii) and (iv)), it 
is enough to prove that for some $\beta \in (0,1)$, $b > 0$, and a petite set $C$
\[
\E_x[ V(X_1) ]= \E [ V(X_1) | X_0 = x ] \leq \beta V(x) + b \ind{ x \in C }.
\]

Using Jensen's inequality and that 
$(u+v)^\alpha \leq u^\alpha + v^\alpha$, $u,v > 0$,
we have
\[
\begin{split}
\E \left[ \left( \sum_{i=1}^x A_i + B \right)^\alpha \right] & \leq 
\E \left[ \left( 
\E \left[ \sum_{i=1}^x A_i + B \,\bigg|\, \mathcal{E} \right] \right)^\alpha \right] \\
& \leq x^\alpha \E m(\xi)^\alpha + \E m^\circ(\xi)^\alpha.
\end{split}
\]
Therefore, there exist $\beta \in (0,1)$, $b>0$, and $x_0 \in \N$ such that
for all $x \in \N$
\[
\E_x[ V(X_1) ]= \E [ V(X_1) | X_0 = x ] \leq \beta V(x) + b \ind{ x \leq x_0 }.
\]
Moreover, the level set $ M=\{x : x \leq x_0 \}$ is small in terminology of Meyn and 
Tweedie~\cite{MeyTwe}. Indeed, for any $x \leq x_0$
\[
P (x, C)   = \p \left( \sum_{i=1}^x A_i + B \in C \right) 
 \geq \p \left( \sum_{i=1}^{x_0} A_i = 0\,; B \in C \right) =: \mu(C)\,.
\]
Since we assumed that the process is subcritical, measure $\mu$ is not trivial, $M$ is 
small, and therefore petite as well. Thus the result follows.
\end{proof}

Geometric ergodicity implies that $(X_n)_n$ is strongly mixing (Meyn and Tweedie
\cite{MeyTwe} or Jones \cite{Jones}), which further
implies by Proposition 1.34 in \cite{Kri} the mixing condition $\mathcal{A}'(a_n)$
(see Condition 2.2 in \cite{BKS}).
That is, there is a sequence $r_n \to \infty$, $r_n = o(n)$, 
such that for any $f: [0,1] \times [0,\infty) \to [0,\infty)$ for which 
there exists a $\delta > 0$ such that $f(x,y) = 0$ whenever $y \leq \delta$, 
we have
\[ %\begin{equation} \label{eq:Aan}
\E \exp \left\{ - \sum_{i=1}^n 
f \Big( \frac{i}{n}, \frac{X_i}{a_n} \Big) \right\} -
\prod_{k=1}^{k_n}
\left( \E \exp \left\{ - \sum_{i=1}^{r_n} 
f \Big (\frac{k r_n}{n}, \frac{X_i}{a_n} \Big) \right\}
\right) \to 0,
\] % \end{equation}
where $k_n = [n/r_n]$.

Roughly speaking, the last condition ensures that the sequence $(X_t)$ 
can be split into blocks of consecutive observations
\[ %\begin{equation}\label{eq:blocks}
C_i = C_i(n) = (X_{(i-1) r_n + 1},  \ldots, X_{i r_n }) \,,
\qquad i =1,2,\ldots, k_n \,,
\] % \end{equation}
which are asymptotically independent.
Individual blocks could be  considered as random  elements of the space
\[
l_0=\{(x_{i})_{i\in\Z}:\lim_{|i|\toi}|x_{i}|=0\}\,,
\]
see also \cite{basrak:planinic:soulier:18}.
This embedding boils down to concatenating infinitely many zeros before and after a 
given block. 
We equip the space $l_0$ with the sup-norm $\|(x_{i})_{i}\|:=\sup_{i \in \Z}|x_{i}|$ 
and with the corresponding Borel $\sigma$-field.

Consider now stationary sequence $(X_t)$, and 
recall  that 
\begin{equation}\label{eq:TailProc}
 \mathcal{L}\left( \left(\frac{X_t}{x} \right)_t \bigg| \, X_0 >x \right)
 \dto \left(Y_t \right)_t\,,
\end{equation}
where the convergence in distribution is to be understood here with
respect to the product topology (which simply corresponds to the convergence of 
finite-dimensional distributions). Recall further that the
tail process has the form 
$Y_{k}  = Y_0 \Theta_{k}$, $k \in \Z$, where $(\Theta_k)$ is two 
sided multiplicative random walk introduced  in \eqref{eq:TailFor} and 
\eqref{eq:TailBack}.
Since the walk has negative drift, $Y_{t} \to 0$ a.s.~for
$|t|\toi$, that is $(Y_t) \in l_0$ with probability 1. 
Moreover, since the anticlustering condition holds by Lemma \ref{lemma:anticluster},
Proposition 4.2 in \cite{basrak:segers:09}
(see also the proof of Theorem 4.3 in \cite{basrak:segers:09}) implies that
\begin{equation}\label{eq:extind}
\theta = \p \big( \sup_{t<0} Y_t < 1 \big) = 
\p \big( \sup_{t>0} Y_t \leq 1 \big)
= \p \big( \max_{t>0} Y_0 m_1 \cdots m_t  \leq 1 \big)
\end{equation}
is strictly positive. 
By Theorem 4.3 and Remark 4.6 in \cite{basrak:segers:09} 
there exist two distributions of  
random elements, $(Z_t)$ and $(Q_t)_t$ say, in $l_0$ such that
\begin{align}\label{eq:Q's_definition}
(Z_t)_t \eind (Y_t)_t  \; \big| \; \sup_{t<0} Y_t \leq 1 \,,
 \quad \mbox{and} \quad (Q_t)_t \eind (Z_t)_t / \max \{ Z_{t} : t \in \Z \} \,.
\end{align}

\subsection{Point process convergence and partial maximum}\label{ssec:PPConv}

Consider now a branching process with immigration in random
environment started from an arbitrary initial distribution $\zeta$.
Recall that $P(x,\cdot)$ denotes the Markov
transition kernel of the sequence $(X_n)$. By  $\zeta P^n$ we represent
the distribution of the random variable $X_n$. 
In the following theorem we use the notion of convergence in distribution
for point process. Following Kallenberg~\cite{Ka17}, we endow the space
of point measures on the the state space
$[0,1]\times(0,\infty)$ by the vague topology. Recall, that (deterministic)
measures $\nu_n$ converge vaguely to $\nu$ in such a topology if 
$\int f \dd \nu_n \to \int f \dd \nu$ for any continuous bounded function
$f$ with a support in some set of the form  $[0,1]\times(x,\infty)$, $x>0$.

 \begin{theorem}\label{thm:mainPP}
Let $(X_t)_{t\geq 0}$ be a  branching process with immigration in random 
environment -- b.p.i.r.e. satisfying \eqref{eq:def-X} with an arbitrary initial 
distribution.
Assume \eqref{eq:main-ass},  then
\begin{equation} \label{main_2}
N_n=\sum_{i=0}^n \delta_{(i/n, X_i/a_n)} \dto
N=\sum_{i}\sum_{j}\delta_{(T_i, P_iQ_{ij})},
\end{equation}
where % $(a_n)$ and $(Z_n)$ are defined as above, and 
\begin{itemize}
\item[i)] $\sum_{i}\delta_{(T_i,P_i)}$ is a Poisson process on 
$[0,1]\times(0,\infty)$ with intensity measure $Leb \times \nu$ where
$\nu(\dd y)= \theta \kappa y^{-\kappa-1} \dd y$ for $y>0$.
\item[ii)] $(({Q_{ij}})_j)$\,, $ i=1,2,\ldots$, is an i.i.d.~sequence of 
elements in $l_0$ independent of $\sum_{i}\delta_{(T_i, P_i)}$ with common distribution 
equal to the distribution of $({Q_j})$ in \eqref{eq:Q's_definition}.
\end{itemize}
\end{theorem}
 
 \begin{remark} \label{rem:ThetaQ}
As we have seen in \eqref{eq:extind}, one can characterize the key constant 
$\theta$ in  the theorem using the tail process of \eqref{eq:TailProc} as
\begin{equation*}
\theta =  \p \big( \sup_{t>0} Y_0 \Theta_{t} \leq 1 \big)
= \p \big( E_0 + \sup_{t>0}  S_t  \leq 0 \big)
\,,
\end{equation*}
where random variable $E_0$ is independent of two sided random walk $S_t = \log 
\Theta_t,\, t \in \Z,$ and has exponential distribution with parameter $\kappa$.

To describe the distribution of the $({Q_j})$ in the theorem,
consider the quotient space $\lo$ of elements in $l_0$ which are shift-equivalent 
(elements $(x_{i})_{i}, (y_{i})_{i} \in l_0$ are shift-equivalent if for
some $j\in\Z$, $(x_{i +j})_{i}=(y_{i})_{i}$),
cf.~\cite[Section 2]{basrak:planinic:soulier:18}.
It is shown in Basrak and Planini\'c \cite{basrak:planinic:19}
that  in $\lo$, $(Q_t) $ has the same distribution as  $(\Theta_t)$  under the 
condition that 
$\Theta_t<1$ for $t<0$ and  $\Theta_t \leq 1$ for $t>0$.
In other words, 
$(Q_t)$ has the same distribution as  $(\exp {S_t})\,,$ with  
random walk $(S_t)$ conditioned on staying strictly negative for $t<0$ and non 
positive for $t>0$. We refer to %Tanaka~\cite{tanaka:1989} and
Biggins~\cite{biggins:2003} 
for more about random walks conditioned in this way.
\end{remark}

 \begin{proof}
Assume first that $(X_t)_{t\geq 0}$ is a stationary branching process
with immigration in random environment -- b.p.i.r.e.~satisfying
\eqref{eq:def-XZ}. The statement of the theorem follows immediately
from Theorem 3.1 in Basrak and Tafro 
\cite{BT16} together with Lemmas \ref{lemma:anticluster} and \ref{lemma:X-ergod}
and discussion following the second lemma.

It remains to prove the convergence of point processes in 
\eqref{main_2} in the case when $X_0$ has an arbitrary initial distribution
$\zeta$. Observe that
by the proof of Lemma~\ref{lemma:X-ergod}, the function $V$ is superharmonic
for the Markov transition kernel $P(x,\cdot)$ outside of the level set
$M$. Moreover, by the last argument in the proof of Lemma~\ref{lemma:X-ergod} 
each level set $\{x : x \leq  r \}$ is petite.
By Theorem 10.2.13 in Douc et al.~\cite{DMS18} the Markov transition kernel 
$P$ is Harris 
recurrent. This, together with Theorem 11.3.1 in \cite{DMS18} shows that for
any initial distribution $\zeta $ of $X_0$, 
\[
\|  \zeta P^m  - \pi  \|_{TV} \longrightarrow 0 \,,
\]
as $m \toi $, where $\zeta P^m$ denotes the distribution of $X_m$.

Take an arbitrary continuous nonnegative function $f$ with support in $[0,1]\times 
(\vep,\infty) >0$ for some $\vep >0$. Consider Laplace functional 
(see Kallenberg \cite[Chapter 4]{Ka17}) of
the point process $N_n$ in \eqref{main_2} under initial distribution 
$X_0 \sim \zeta$
\[
\E_\zeta \exp \bigg(- \sum_{i=1}^n f(i/n, X_i/a_n) \bigg) 
= \E_\zeta  \exp \bigg(- \sum_{i=m+1}^n f(i/n, X_i/a_n) \bigg) 
  + r_{n,m}\,,
\]
where for  $m$ fixed, $r_{n,m} \to 0 $ as $n \toi$, because $X_i/a_n \to 0$
a.s.~for $i =1,2$, $\ldots, m$. Thus
\begin{align*} 
& \E_{\zeta} \exp \bigg(- \sum_{i=m+1}^n f(i/n, X_i/a_n) \bigg) \\
& = \E_{\zeta P^m}  
\exp \bigg(- \sum_{i=1}^{n-m} f((i+m)/n, X_{i}/a_n) \bigg) 
 =:  \dint_\N H_{n} \, \dd \zeta P^m
\,,
\end{align*}
for a suitably chosen function $H_n$, which is nonnegative and bounded by 1. 
Similarly, for the stationary Markov chain $(X_n)$
\begin{align*}
&\E_\pi  \exp \left(- \sum_{i=1}^n f(i/n, X_i/a_n) \right) 
= \dint_\N H_{n} \, \dd \pi 
+ r'_{n,m}\\
&\longrightarrow 
\E  \exp \bigg(- \sum_{i} \sum_j f(T_i, P_i Q_{ij}) \bigg).  
\end{align*}
Observing that 
$\left| \int_\N H_{n} \, \dd \zeta P^m -\int_\N H_{n} \, \dd \pi 
 \right| \leq
 \|  \zeta P^m  - \pi  \|_{TV} \,,
$ uniformly over $n \in \N$,  and
$r'_{n,m} \to 0 $ for fixed $m$ as $n \toi$, yields the statement.
\end{proof}

We consider next partial maxima of the sequence $(X_t)_{t\in \N}$, namely we define
$M_n= \max \{X_1,\ldots, X_n \}$, for any $n\geq 1$. Observe that event 
$\{M_n/a_n \leq x\}$ corresponds to the event 
$\{N_n([0,1]\times (x,\infty)) = 0\}$.
Moreover, for any $x>0$, the limit point process $N= \sum_i\sum_j  \delta_{T_i,P_iQ_{ij}} 
$ has probability 0 of hitting 
the boundary of the set $[0,1]\times (x,\infty)$, thus
$\p (N_n([0,1]\times (x,\infty)) = 0)  \to 
\p (N([0,1]\times (x,\infty)) = 0 )$.
 However, by \eqref{eq:Q's_definition}, $(Q_{ij})$ in \eqref{main_2} satisfy $Q_{ij} \leq 1$
with at least one point exactly equal to $1$. 
Thus $\p (N([0,1]\times (x,\infty)) = 0 ) = 
\p (\sum_i   \delta_{T_i,P_i }([0,1]\times (x,\infty)) = 0)$. 
Therefore,   for any initial distribution of $X_0$, partial maxima 
converge to a rescaled Fr\'echet distribution.

\begin{corollary} \label{cor:Max}
 Let $(X_t)_{t\geq 0}$ be a  branching process with immigration in random environment -- 
b.p.i.r.e.~satisfying \eqref{eq:def-X} with an arbitrary initial distribution.
Suppose that  \eqref{eq:main-ass} holds. Then for any $x \geq 0$
\[
\p\left(\frac{M_n}{a_n} \leq x \right) \to e^{-\theta x^{-\kappa}} 
\quad \text{as } \ n \toi.
\]
\end{corollary}

\section{Partial sums}\label{sec:PartialSums}

Denote by $(X_t)$ a stationary branching process with immigration in a random environment 
and assume that at least one of the conditions in \eqref{eq:main-ass} holds.
To derive limit theorem for the partial sums from the 
point process convergence is immediate for $\alpha \in (0,1)$, but needs
an extra condition for $\alpha \in [1,2)$.

%%% started from an arbitrary initial distribution $\zeta$.  

\begin{lemma}[Vanishing small values] \label{lem:vsv}
Assume that $\kappa \in [1,2)$. Then for any $\varepsilon > 0$
\[
\lim_{\gamma \downarrow 0} \limsup_{n \to \infty}
\p \left( 	
 \left| \sum_{i=1}^n \left[ 
X_i \ind{ X_i \leq a_n \gamma } - \E [X_0 \ind{X_0 \leq a_n \gamma}]
\right] \right| > a_n \varepsilon \right) = 0.
\]
\end{lemma}

\begin{proof}
Choose $\alpha < 1$ such that $1 + \alpha > \kappa$. 
(For $\kappa > 1$ we may choose  $\alpha = 1$.)
Lemma \ref{lemma:X-ergod}  holds with $V(x) = 1 + x^\alpha$. 
Let $\varepsilon = 1 - \alpha$. Then, 
\[
h(x) : = \frac{x}{(a_n \gamma)^\varepsilon} \ind{x \leq a_n \gamma}
\leq x^\alpha, \quad \text{for } x \in \N.
\]
Thus $|h| \leq V$, therefore by Lemma \ref{lemma:X-ergod}, 
for some $\rho \in (0,1)$ and $C > 0$ we have
\[ %\begin{equation} \label{eq:geo-ergod-V}
\left| 
\E \left[ 
\frac{X_i}{(a_n \gamma)^\varepsilon} \ind{ X_i \leq a_n \gamma} \, \big|
 X_0 = m \right] -
\E \frac{X_0}{(a_n \gamma)^\varepsilon} \ind{X_0 \leq a_n \gamma} \right| 
\leq C (1 + m^\alpha) \rho^i.
\] % \end{equation}
Thus
\begin{equation} \label{eq:cov-bound}
\begin{split}
& \E  \left[ X_i \ind{X_i \leq a_n \gamma} X_0 
\ind{X_0 \leq a_n \gamma} \right] \\
& = \sum_{m \leq a_n \gamma} m \p ( X_0 = m)  
(a_n \gamma)^\varepsilon \E \left[ \frac{X_i}{(a_n \gamma)^\varepsilon} 
\ind{ X_i \leq a_n \gamma}\, \big| \,  X_0 = m \right] \\
& \leq \sum_{m \leq a_n \gamma} m \p ( X_0 = m)  
(a_n \gamma)^{\varepsilon} \left( \E \frac{X_0}{(a_n \gamma)^\varepsilon} 
\ind {X_0 \leq a_n \gamma} + C ( 1 + m^\alpha) \rho^i \right) \\
& \leq \left[ \E X_0 \ind {X_0 \leq a_n \gamma} \right]^2 +
2 C \rho^i (a_n \gamma)^\varepsilon 
\E X_0^{2-\varepsilon} \ind {X_0 \leq a_n \gamma}.
\end{split}
\end{equation}
For $\beta > \kappa$
by the regular variation and Karamata's theorem as $u \to \infty$
\begin{equation} \label{eq:Karamata-2}
\E X_0^{\beta} \ind{ X_0 \leq u }  \sim 
\frac{\beta}{\beta-\kappa} u^{\beta} \overline F(u).
\end{equation}
Substituting back into \eqref{eq:cov-bound}, and using
that $\overline F(a_n \gamma) \sim \gamma^{-\kappa} n^{-1}$ we have
\[
\cov \left( X_0 \ind{X_0 \leq a_n \gamma}, 
X_i \ind{X_i \leq a_n \gamma} \right)
\leq C \rho^i \frac{a_n^2}{n} \gamma^{2 - \kappa}.
\]
Therefore, using the stationarity, \eqref{eq:cov-bound}, and \eqref{eq:Karamata-2},
we obtain
\[
\begin{split}
& \E \left( \sum_{i=1}^n \left[ 
X_i \ind{X_i \leq a_n \gamma} - \E [X_0 \ind{X_0 \leq a_n \gamma}]
\right] \right)^2 \\ 
& = n \var \left( X_0 \ind{X_0 \leq a_n \gamma} \right) \\
& \qquad + \sum_{i=1}^{n-1} 2 (n-i) \cov \left( X_0 \ind{X_0 \leq a_n \gamma}, 
X_i \ind{X_i \leq a_n \gamma} \right) \\
& \leq n \var \left( X_0 \ind{X_0 \leq a_n \gamma} \right)
+ C \sum_{i=1}^{n-1} (n-i) \rho^i \frac{a_n^2}{n} \gamma^{2 - \kappa} \\
& \leq C a_n^2 \gamma^{2- \kappa}.
\end{split}
\]
Therefore, the claim follows from Chebyshev's inequality.
\end{proof}

\begin{remark}
In fact, we proved condition (9) in Davis \cite{Davis83}. 
\end{remark}

\begin{theorem} \label{thm:CLT}  
Let $(X_t)_{t\geq 0}$ be a  branching process with 
immigration in random environment -- b.p.i.r.e.~satisfying \eqref{eq:def-X} with an 
arbitrary initial distribution.  Assume \eqref{eq:main-ass} and
denote by $(b_n)$ a sequence of real numbers given by
$$
b_n = 0,\quad  \kappa < 1, \quad \mbox{ and } \quad 
b_n = n \E \bigg( \frac{X_{\infty}}{a_{n}} 
\ind{ \frac{X_{\infty}}{a_{n}} \le 1  } \bigg), 
\quad \mbox{ for }  \kappa \in [1,2).
$$
Then
\begin{equation}\label{eq:StLim}
V_{n} =
\sum_{k=1}^{n} \frac{X_{k}}{a_{n}} - b_n \dto V, \qquad n \to \infty,
\end{equation}
where $V$ is a $\kappa$-stable random variable.
For $\kappa > 2$, as $n \to \infty$
\[
\frac{1}{\sqrt{n} \sigma}  \sum_{j=1}^n (X_i - \E X_\infty ) \dto Z,
\]
where $Z$ is a standard normal random variable, and 
$\sigma^2 = \frac{1 + \E A}{1  - \E A} \var (X_\infty) > 0$.
\end{theorem}

\begin{remark}\label{rem:Stable}
Recall that under the conditions of Theorem \ref{thm:nonarith}
for the normalizing sequence we may choose
 $a_n = (Cn)^{1/\kappa}$, in general we have
$a_n = \widetilde \ell(n) n^{1/\kappa}$, for some 
slowly varying $\widetilde \ell$, as observed  after \eqref{eq:def-a}.
% under the conditions of Theorem \ref{thm:nonarith}}
% for the normalizing sequence we may choose
% any $a_n \sim (Cn)^{1/\kappa}.$
Concerning the centering constants $(b_n)$, recall that the
the law of $X_\infty$ represents the unique
stationary distribution for b.p.i.r.e.~satisfying \eqref{eq:def-X}.
For $\kappa\in (1,2)$, the mean of $X_\infty$  is finite,  so one could 
substitute centering constants $(b_n)$  by $ n \E X_\infty /a_n$ to show
\begin{equation} \label{eq:DiffCent}
\frac{1}{a_n}   \sum_{j=1}^n (X_i - \E X_\infty ) \dto V - 
\frac{\kappa}{\kappa-1}\,.
\end{equation}
Thus, we again get a $\kappa$-stable limit with different location parameter,
cf.~Remark 3.1 in \cite{DH95}. 
Moreover, under conditions of Theorem~\ref{thm:nonarith}, in the case $\kappa=1$,
 $b_n \sim C^{-1} \log n $.
\end{remark}

\begin{remark} \label{rem:StableLaw}
One can also describe the limiting  stable distribution in 
\eqref{eq:StLim} in terms of parameters $\kappa\,, \theta$ and the distribution of 
$(Q_j)$, see  Remark 3.2 in \cite{DH95} for details.
In the case $\kappa<1$ for instance, the limiting random variable $V$  has a
characteristic function of relatively simple form
\begin{equation} \label{eq:ChFun}
\E e^{i  t V}  = 
\exp \left( 
- d |t|^\kappa \left( 1- i\, \text{sgn}(t) \tan \frac{\pi \kappa}{2}\right) \right) \,,
\end{equation}
where $d = \theta \, \Gamma(1-\kappa) \,
\E(\sum_j Q_j)^\kappa  \cos (\pi \kappa /2)$, which is known to be  
finite since for  
$\kappa\leq 1$,  $\E(\sum_j Q_j)^\kappa \leq  \E (\sum_j Q_j^\kappa ) < \infty$,
see \cite[Remark 3.2, Theorem 3.2]{DH95}.

For $\kappa \in (1,2)$, multiplicative random walk $(\Theta_j)$ given by 
\eqref{eq:TailFor} and
\eqref{eq:TailBack} satisfies
$\E (\sum_j \Theta_j)^{\kappa-1}  \leq \E (\sum_j \Theta_j^{\kappa-1} ) < 
\infty$. 
By Theorem 5.5.1 in Kulik and Soulier~\cite{KS19}, this yields
$\E(\sum_j Q_j)^\kappa < \infty$. Applying Proposition 8.3.2 and Theorem 
8.4.3   in 
\cite{KS19},
and observing that they use $ n \E X_\infty /a_n$ as centering constants,
one can deduce that
\begin{equation*} 
\E e^{i  t V}  = 
\exp \left( 
- d |t|^\kappa \left( 1- i\, \text{sgn}(t) \tan \frac{\pi \kappa}{2}\right)
+ i ct \right) \,,
\end{equation*}
with the same expression for the scale parameter $d$ and with  $c= 
\kappa/(\kappa-1)$, cf. \eqref{eq:DiffCent}.

For $\kappa =1$, $(\Theta_j)$ satisfies
$\E [\log (\sum_j \Theta_j) ] \leq \E (\sum_{j} \Theta_j)^{1/2}  < \infty$. 
By Chapter 5 in Kulik and Soulier~\cite{KS19} (see Problem 5.31 therein), this 
implies
$\E(\sum_j Q_j \log Q^{-1}_j )< \infty$. This time,  Theorem 8.4.3   in 
\cite{KS19} yields
$$
\E e^{i  t V}  = 
\exp \left( 
- d |t| \left( 1+ i\, \frac{2}{\pi} \text{sgn}(t) \log |t| \right)  + i c t
\right) \,,
$$
where $d = \theta \frac{\pi}{2} \E(\sum_j Q_j)$ while an expression for the location 
parameter $c$ can be found in Proposition 8.3.3 of \cite{KS19}.

Note finally,  that for $\kappa > 2$ one can easily check that the 
conditions of Theorem 19.1 in Billingsley \cite{Billingsley} hold, therefore
the functional version of the CLT also holds.
\end{remark}  

\begin{proof}
Assume first that $\kappa < 2$.
As in the proof of Theorem~\ref{thm:mainPP}, we first assume that 
$(X_t)$ is a stationary b.p.i.r.e.~process satisfying the conditions of the theorem. The 
claim then follows directly from Theorem~\ref{thm:mainPP} and  Theorem 3.1 in Davis and 
Hsing~\cite{DH95}. Observe that the condition (3.2) therein 
follows from Lemma~\ref{lem:vsv}.

To prove the theorem when $X_0$ has an arbitrary initial distribution we can use 
 a similar argument
as in the proof of Theorem~\ref{thm:mainPP}.
 Observe first, that 
$$
V_n = \sum_{k=1}^{n} \frac{X_{k}}{a_{n}} - b_n  
 =  \sum_{k=1}^{n-m} \frac{X_{k+m}}{a_{n}} - b_n + r_{n,m}
$$
 with $r_{n,m} \to 0$ in probability as $n \toi$ for any fixed $m$. 
 Denote by
 $\zeta P^m$  the distribution of $X_m$, and note that for any $s \in \R$ 
\begin{align*}
 &	\E_\zeta \left[ \exp \left( i s \left(
	\sum_{k=1}^{n-m}  \frac{X_{k+m}}{a_{n}} - b_n\right)
	\right) \right]  =
	\E_{\zeta P^m} \left[ \exp \left( i s \left(
	\sum_{k=1}^{n-m} \frac{X_{k}}{a_{n}}  - b_n\right)
	\right) \right]\\
	&
	=
	\E_{\pi} \left[ \exp \left( i s \left(
	\sum_{k=1}^{n-m} \frac{X_{k}}{a_{n}}  - b_n\right)
	\right) \right]   + u_{n,m}\,,
\end{align*}
where for all $n$
\[
 u_{n,m} \leq \|  \zeta P^m  - \pi  \|_{TV} \,,%\longrightarrow 0 \,,
\]
with the right hand side tending to 0 as $m \toi $. Observe now that a stationary 
b.p.i.r.e. $(X_t)$  satisfies   $r'_{n,m} 
=\sum_{k=n-m+1}^{n} \frac{X_{k}}{a_{n}}  \to 0$ in probability as $n \toi$. Therefore, it 
also satisfies
$$
\sum_{k=1}^{n-m} \frac{X_{k}}{a_{n}} - b_n = \sum_{k=1}^{n} \frac{X_{k}}{a_{n}} -
b_n   -r'_{n,m} \dto V\,.
$$

\smallskip

Let now $\kappa > 2$. Then $\var (X_\infty)< \infty$, and the result follows 
from the standard Markov chain theory.
To apply Theorem 1 (i) in Jones \cite{Jones} we have to prove the drift 
condition as in Lemma \ref{lemma:X-ergod} with $V(x) = 1 + x^2$.
Simple calculation gives 
\[
\E \bigg( \sum_{i=1}^x A_i + B \bigg)^2 
= x^2 \E m(\xi)^2 + x \left( \E A^2 - \E m(\xi)^2 + 2 \E AB
\right) + \E B^2.
\]
Since $\E m(\xi)^2 < 1$, there exist $\beta \in (0,1)$, $b>0$, and 
$x_0 \in \N$ such that for all $x \in \N$
\[
\E_x [ V(X_1) ]= \E [ V(X_1) | X_0 = x ] \leq \beta V(x) + b \ind{ x \leq x_0 }.
\]
Moreover, the level set $ M=\{x : x \leq x_0 \}$ is small, so 
the conditions in Theorem 1 (i) in Jones \cite{Jones} hold.

We only have to check that 
$\sigma^2 = 
\var_\pi(X_0) + 2 \sum_{t=1}^\infty \mathrm{Cov}_\pi(X_0, X_t) > 0$.
Using decomposition \eqref{eq:stat-decomp} from the proof of 
Lemma \ref{lemma:anticluster} we see that
\[
\cov_\pi ( X_0, X_t ) = \cov_{\pi} 
\Big( X_0, \sum_{i=1}^{X_0} \widetilde  A_i^{(t)} \Big),
\]
where $\widetilde A_i^{(t)}$ stands for the number of descendants 
in generation $t$ of a single individual in generation 0. Thus
\[
\E \left( X_0 \sum_{i=1}^{X_0} \widetilde A_i^{(t)} \right) =
( \E A )^t \E X_0^2. 
\]
Therefore
\[
\cov_\pi ( X_0, X_t ) = \var_\pi(X_0) \, (\E A)^t, 
\]
and the formula $\sigma^2 = \var_\pi(X_0) (1 + \E A) / (1 - \E A)$
follows. 

To show that $\sigma > 0$ we have to prove that the stationary solution
cannot be deterministic. Indeed, the fixed point equation 
\eqref{eq:fixed} (or \eqref{eq:fixed2}) has no deterministic solution,
since Cram\'er's condition \eqref{eq:ass-cramer} implies $\p (A > 1) >0$,
thus $\p (\sum_{i=1}^m A_i + B >  m) > 0$ for any  positive integer $m$.
\end{proof}

\section{Random walks in a random environment}\label{sec:RWRE}

Connection between random walks in a random environment
and branching processes with immigration was made already
in the seminal paper by
Kesten et al.~\cite{KKS}.
To describe the setup of their model, we can use again 
a sequence of i.i.d.~random variables
$\xi\,,(\xi_t)_{t \in \Z}$ with values in the interval $(0,1)$.
%Denote by $\xi'_t= 1- \xi_t$, and 
Consider the process
$W_0=0$ and 
\begin{equation} \label{eq:rwre}
	W_t = W_{t-1} + \eta_{W_{t-1},t}  
\end{equation}
where conditioned on the environment $\mathcal{E}$, $\eta_{x,t}$'s
are independent random variables taking value $+1$ with probability
$\xi_x$ and value $-1$ with probability $\xi'_x = 1 -\xi_x$.
The process $(W_t)$ is called random walk in random environment -- r.w.r.e.

Following \cite{KKS} we assume throughout that
 \begin{equation} \label{eq:rwre-drift}
 \E \log \frac{\xi'}{\xi} < 0
 \end{equation}
and 
 \begin{equation} \label{eq:rwre-cramer}
  \E \left( \frac{\xi'}{\xi} \right)^\kappa =1
  \quad \mbox{ and } \quad 
  \E \left( \frac{\xi'}{\xi} \right)^\kappa  
  \log \frac{\xi'}{\xi}
 < \infty\,,
 \end{equation}
for some $\kappa >0$. 
The first of these conditions ensures that $(W_t)$  drifts to $+\infty$
and the second one corresponds to \eqref{eq:ass-cramer}.
As observed in \cite{KKS}, for the asymptotic analysis of $(W_t)$,
it is crucial to understand asymptotic behavior of the random variables
$$
 T_n = \min  \{ t : W_t = n \}, \quad n \geq 1\,,
$$
which are all finite a.s. 
Observe that on the way to the state $n$, process $(W_t)$ visits
each state $k=0,1,\ldots, n-1$ at least once. 
For $i=1,\ldots, n$  denote 
$$
 L^n_i = 1 + U^n_i = 1 + \#  \{ \mbox{visits to the state } n- i  
 \mbox{ from the right} \mbox{ before }  T_n \}.
$$
Note, until time $T_n$ each of the visits to $n-i$ from the right is
canceled by one movement back to the right. Note 
$U^n_i$ are well defined and possibly different from 0 also
for $i>n$. Observe next that the total number of visits by the
process $(W_t)$ to the left of 0, say 
$2 U_{\infty} $, is a.s.~finite. Thus,
$ R_n = 2 \sum _{i>n}   U_i^n \leq  2 U_{\infty}$\,.
Thus in
$$
   T _n = n + 2  \sum _{i\geq 1 }   U_i^n
   = n + 2  \sum _{i=1 }^n   U_i^n + R_n,
$$
the term $R_n$ remains bounded as $n\toi$.

Observe next that  $U_1^n = 0$, so that $L_1^n = 1$  and the
sequence $L_i^n$ evolves as follows
$$
 L_2^n = G_{n,1,1} + 1,
$$
where $G_{n,1,1}$ represents the number of right visits to $n-2$ from $n-1$ before 
eventual move to the right. Conditioned on $\mathcal{E}$,
$G_{n,1,1}$ has a geometric distribution with mean
$\xi'_{n-1} /\xi _{n-1}$. Similarly, for each $i \geq 2$
\begin{equation} \label{eq:Lbpire}
L_{i}^n = \dsum_{j=1}^{ L_{i-1}^n}  G_{n,i-1,j} + 1\,,
\end{equation}
where conditionally on $\mathcal{E}$,
$G_{n,{i-1},j}$, $j=1,2,\ldots$, are i.i.d.~with a geometric distribution with mean
$\xi'_{n-i+1} /\xi _{n-i+1}$.

The finite sequences 
$$
L^n_1,L^n_2,\ldots, L^n_k,\ldots, L^n_n,
$$ 
represent a special case of branching process in random environment.
Moreover,  the initial part $L^n_1,L^n_2,\ldots, L^n_k$
has the same distribution for each $n\geq k$, hence if we want to understand the limiting distribution of
$ T_n  = 2  \sum _{i=1 }^n   L_i^n -n +R_n\,,$
we can simply skip the index $n$ in \eqref{eq:Lbpire}
and analyze $\sum _{i=1 }^n   L_i$, 
which is exactly the content of  Theorem~\ref{thm:CLT}. Denote in the sequel
by $L_\infty$ the random variable which has the stationary distribution
of the Markov chain in \eqref{eq:Lbpire}.
The following theorem corresponds to the part of the main theorem in \cite{KKS}
concerning the case $\kappa<2$. The corresponding statement 
for $\kappa >2 $ follows from
Theorem~\ref{thm:CLT} as well, in that case $T_n$ after centering and
normalization with $\sqrt n$ converges to a normal distribution. Observe
that in the case $\kappa = 2$,  $T_n$ converges to a normal limit again,
but additional care has to be taken about normalizing and centering constants,
see \cite{KKS}.

\begin{theorem} \label{thm:RWRE} 
Let $(W_t)_{t\geq 0}$ be a  random walk  in random environment  satisfying 
\eqref{eq:rwre}.
Suppose that the conditions of \eqref{eq:rwre-drift}
and \eqref{eq:rwre-cramer} hold for some
$\kappa>0$. 
Suppose further that the law of $\log ( {\xi'_i}/{\xi_i})$ is nonarithmetic. 
\begin{itemize}
\item[(i)]  For $\kappa \in (0,1)$ as $n\toi$
$$
\frac{1}{n^{1/\kappa}} T_n \dto \widetilde V,
$$
where $\widetilde V $ has a strictly positive $\kappa$--stable distribution,
while, as $t \toi$
$$
\frac{1}{t^{\kappa}} W_t \dto \widetilde V ^{-\kappa}. 
$$

\item[(ii)]  For $\kappa =1 $
$$
\frac{1}{n}  T_n - 2 C  b_n \dto \widetilde V,
$$
where $b_n \sim  \log n$ and 
$\widetilde V $ has a $1$-stable distribution. 
Moreover, as $t \toi$
$$
\frac{(\log t)^2}{t} ( W_t - \delta(t) ) \dto  
\frac{  - \widetilde V }{(2C)^{2}},
$$
with $\delta(t) \sim t / (2 C \log t)$.

\item[iii)]  For $\kappa \in (1,2)$
$$
\frac{1}{n^{1/\kappa}} ( T_n - n  (2 \E L_\infty - 1) ) \dto \widetilde V.
$$
Moreover, as $t \toi$
$$
\frac{1}{t^{1/\kappa}} 
\left( W_t - \frac{t}{2 \E L_\infty - 1} \right) \dto
\frac{  - \widetilde V }{(2 \E L_\infty - 1)^{1+1/\kappa}}.
$$
\end{itemize}
\end{theorem}

\begin{proof}
Recall that 
$T_n 
\eind  \sum _{i=1 }^n  (2 L_i -1)  +R_n\,,$
with random variables  $R_n$ a.s.~bounded and where $(L_n)$ represents a special 
b.p.i.r.e.~which is initialized at one,   
with progeny conditionally geometric 
given the environmental variables $\xi_j$'s and with the immigration
identically equal to $1$. Recall further that the conditions of 
Theorems~\ref{thm:nonarith}  and \ref{thm:CLT} are met by assumptions 
\eqref{eq:rwre-drift}, \eqref{eq:rwre-cramer} and Remark \ref{rem:PoiGeo}. 

Define now
$$
 c_n = \left\{
 \begin{array}{ll}
  0 & \mbox{ for } \kappa \in (0,1)\,,\\
    2{\E L_\infty \ind{ L_{\infty}\le (Cn)^{1/\kappa}  } } -1  & \mbox{ for } \kappa = 1\,,\\
   2{\E L_\infty} -1  & \mbox{ for }\kappa \in (1,2). 
 \end{array}
 \right.
$$
Theorem~\ref{thm:CLT} and  Remark~\ref{rem:Stable} imply that
${2^{-1} (Cn)^{-1/\kappa}} \sum _{i=1 }^n  (2 L_i -1 - c_n )$
converges in distribution to  a $\kappa$-stable random variable for all 
$\kappa \in (0,2)$.
After  multiplication by $2 C^{1/\kappa}$, this yields
$$
  \frac{1}{n^{1/\kappa}} ( T_n  - n c_n)  \dto \widetilde V \,,
$$
for a certain $\kappa$-stable random variable $ \widetilde V$.
Using the particular form of constants $c_n$  gives  the  centering
sequence in each of the three cases. 
The convergence in distribution of 
the r.w.r.e.~$W_t$ now follows as in Kesten et al.~\cite{KKS}.
\end{proof}

\begin{remark}
Observe that the characteristic function, and the exact distribution therefore,
of the limit $\widetilde V$ in all three statements of the theorem
can now be read out from Remark~\ref{rem:StableLaw}. In particular, 
location and scale parameters of the stable law
$\widetilde V$ are now determined in terms of the 
values $\kappa$, $\theta$ and the conditional multiplicative random walk $(Q_j)$.
\end{remark}

The analogous results hold under the second condition in \eqref{eq:main-ass}. We translate 
this  to the RWRE setup in the following theorem. 
Recall first that normalizing sequence $(a_n)$ from \eqref{eq:def-a} 
satisfies $a_n = n^{1/\kappa} \widetilde \ell(n)$ for some slowly varying function 
$\widetilde \ell$, and extend this to a function
$a(t) = t^{1/\kappa} \widetilde \ell(t)$ on $(0,\infty)$.

\begin{theorem} \label{thm:RWRE2}
Let $(W_t)_{t\geq 0}$ be a  random walk in random environment satisfying 
either condition (i) or (ii) in Theorem \ref{thm:nonarith2} 
for $\kappa \in (1,2)$ with 
$m(\xi) = \xi' / \xi$.
Suppose further that the law of $\log ( {\xi'_i}/{\xi_i})$ is nonarithmetic. 
Then
$$
\frac{1}{a_n} ( T_n - n  (2 \E L_\infty - 1) ) \dto \widetilde V
\quad \text{as } \, n \to \infty.
$$
Moreover, as $t \toi$
$$
\frac{1}{a(t)} 
\left( W_t - \frac{t}{2 \E L_\infty - 1} \right) \dto
\frac{  - \widetilde V }{(2 \E L_\infty - 1)^{1+1/\kappa}}.
$$
\end{theorem}

\begin{proof}
The first part follows again from Theorem \ref{thm:CLT}, while the 
second part follows from the inverse relation between $T_n$ and $W_t$;
see (2.38) in \cite{KKS}. We omit the details.
\end{proof}

\begin{example}
Assume that the random variable $\xi$ 
satisfies $\E (\xi'/\xi)^\kappa = 1$, and has a
 density function $f$ such 
that   for all $u > 0$ small enough
\[
f(u) = u^{\kappa - 1} \alpha \frac{\log \log u^{-1}}{(\log u^{-1} )^{\alpha + 1}},
\]
where $\kappa \geq 1$, $\alpha \in (1/2, 1)$. Then 
straightforward calculation shows that
\[
\E \ind{\log \frac{1-\xi}{\xi} > x} \left( \frac{1-\xi}{\xi} \right)^\kappa 
= \int_0^{(1+e^x)^{-1}} \left( \frac{1-u}{u} \right)^\kappa  f(u) \dd u 
\sim \frac{\log x}{x^\alpha}.
\]
Thus, the conditions of the theorem above holds, and for $M$ in \eqref{eq:M-def}
\[
 M(x) \sim \frac{x^{1-\alpha} \log x}{1- \alpha}.
\]
Therefore by Theorem~\ref{thm:nonarith2}, the stationary distribution $L_\infty$ satisfies
\[
\p (L_\infty > x) \sim 
\frac{C (1-\alpha)}{x^\kappa  (\log x)^{1-\alpha} \log \log x}.
\]
Then 
\[
a_n \sim [\kappa^{1-\alpha} C (1-\alpha)]^{1/\kappa}
\frac{n^{1/\kappa}}{(\log n)^{(1-\alpha)/\kappa} (\log \log n)^{1/\kappa}}.
\]
\end{example}

According to Kesten et al.~\cite{KKS} even before they gave the proof of the theorem, it 
was conjectured by A.N.~Kolmogorov and F.~Spitzer that $T_n$ might exhibit the behavior 
described above. The intuitive reason behind this observation may be the existence of so-called traps between 0 
and $n$, i.e.~sites
$j\in \{0,\ldots,  n\}$ where corresponding $\xi_j$ is atypically small,  which
makes it very difficult for the random walk $(W_t)$ to cross over to the right.
It is interesting that our other main result, Theorem~\ref{thm:mainPP}, provides a very 
simple argument characterizing the asymptotic distribution of the worst of 
such traps. 
Denote now for $k < n$
by
$$
V^n_k = \# \{\mbox{crossings over the edge } (k,k+1) \mbox{ before }
T_n \}\,.
$$
Clearly, for $k= 0,\ldots, n-1$, we have  $V_k=1+ 2 U^n_{n-k}$ and
while for $k<0$  $V_k=2 U^n_{n-k}$. 
Observing that $\max_{k<0}  V^n_k $ remains bounded a.s.~again,
from  Corollary~\ref{cor:Max}  we can deduce the following result
concerning the most visited edge until time $T_n$.

\begin{corollary}
Under the assumptions of Theorem~\ref{thm:RWRE} or Theorem~\ref{thm:RWRE2}
\[
\p \left(\frac{ \max_{k<n}  V^n_k }{ 2 a_n} \leq x \right) \to
e^{-\theta x^{-\kappa}} \,,
\]
as $n \toi$.
\end{corollary}

Recall that by Remark~\ref{rem:ThetaQ}, $\theta$ in either case can be obtained as 
\begin{equation*}
\theta  
= \p \bigg(E_0 + \sup_{t>0} \dsum_{i=1}^t \log ({\xi'_i}/{\xi_i}) 
\leq 0 \bigg)
\,,
\end{equation*}
where $E_0$ stands for an exponential random variable with parameter 
$\kappa$ independent of the environment sequence $(\xi_j)$.

\bigskip

\noindent \textbf{Acknowledgement.}
We are grateful to Dariusz Buraczewski for discussion on the problem and pointing 
out reference \cite{Afan2001}. 
This work is in part financed within the Croatian-Swiss Research Program of the Croatian Science Foundation and the Swiss National Science Foundation -- grant CSRP 018-01-180549.
P\'{e}ter Kevei is supported by the
J\'{a}nos Bolyai Research Scholarship of the Hungarian Academy of Sciences, by
the NKFIH grant FK124141, and by the EU-funded Hungarian grant
EFOP-3.6.1-16-2016-00008.
% Wachtel ?

\end{document}